\documentclass[review]{siamltex}
\usepackage[numbers,sort&compress]{natbib}
\usepackage{graphicx,latexsym} 
\usepackage{amssymb,amsmath}
\usepackage{verbatim}
\usepackage{bbm}
\usepackage{mathtools}
\usepackage{braket}
\usepackage{xcolor}
\usepackage{sansmath}
\usepackage{caption}
\usepackage{subfig}
\usepackage[utf8]{inputenc}
\usepackage{fancyhdr}
\usepackage{algpseudocode}
\usepackage{algorithm}
\usepackage[mathscr]{euscript}

\usepackage{lineno}

\newtheorem{remark}[theorem]{Remark }

\algdef{SE}[SUBALG]{Indent}{EndIndent}{}{\algorithmicend\ }%
\algtext*{Indent}
\algtext*{EndIndent}

\newcommand{\Emat}{\mathsf{E}}
\newcommand{\Amat}{\mathsf{A}}
\newcommand{\Bmat}{\mathsf{B}}
\newcommand{\Ymat}{\mathsf{Y}}
\newcommand{\Umat}{\mathsf{U}}
\newcommand{\Vmat}{\mathsf{V}}
\newcommand{\Vcal}{\mathcal{V}}
\newcommand{\Qmat}{\mathsf{Q}}

\newcommand{\Rmat}{\mathsf{R}}
\newcommand{\Sigmamat}{\mathsf{\Sigma}}

\newcommand{\Ecal}{\mathcal{E}}
\newcommand{\Scal}{\mathcal{S}}
\newcommand{\Kcal}{\mathcal{K}}
\newcommand{\Dcal}{\mathscr{D}}
\newcommand{\Pcal}{\mathcal{P}}
\newcommand{\Ycal}{\mathcal{Y}}
\newcommand{\Lcal}{\mathcal{L}}

\newcommand{\vep}{\varepsilon}

\newcommand{\Null}{\text{Null\,}}

\newcommand{\rd}{\mathrm{d}}

\newcommand{\wt}[1]{\widetilde{#1}}

\newcommand{\R}{\mathbb{R}}

\graphicspath{{../figures/}}

\title{A Low-Rank Schwarz Method for Radiative Transfer Equation with
  Heterogeneous Scattering Coefficient\thanks{The work of JL is
    supported in part by the National Science Foundation via grant
    DMS-1454939. The work of KC, QL, and SW is supported in part by
    the National Science Foundation via grant 1740707. The work of SW
    is further supported in part by National Science Foundation grants
    1628384 and 1634597; Subcontract 8F-30039 from Argonne National
    Laboratory; and Award N660011824020 from the DARPA Lagrange
    Program. The work of KC and QL is further supported in part by
    Wisconsin Data Science Initiative and National Science Foundation
    via grant DMS-1750488, and DMS-1107291: RNMS KI-Net.}}

\author{Ke Chen\thanks{Mathematics Department, University of Wisconsin-Madison, Madison, WI 53706 ({\tt kchen222@wisc.edu})} 
	\and Qin Li\thanks{Mathematics Department and Discovery Institute, University of Wisconsin-Madison, Madison, WI 53706 ({\tt qinli@math.wisc.edu})} 
	\and Jianfeng Lu\thanks{Department of Mathematics, Department of Physics, and Department of Chemistry, Duke University, Durham, NC 27708 ({\tt jianfeng@math.duke.edu}) } 
	\and Stephen J. Wright\thanks{Computer Sciences Department, University of Wisconsin, Madison, WI 53706 ({\tt swright@cs.wisc.edu})}
}

\begin{document}

\maketitle
 
\begin{abstract}
  Random sampling has been used to find low-rank structure and to
  build fast direct solvers for multiscale partial differential
  equations of various types. In this work, we design an accelerated
  Schwarz method for radiative transfer equations that makes use of
  approximate local solution maps constructed offline via a random
  sampling strategy. Numerical examples demonstrate the accuracy,
  robustness, and efficiency of the proposed approach.
\end{abstract}



\pagestyle{myheadings}
\thispagestyle{plain}
\markboth{Ke Chen, Qin Li, Jianfeng Lu and Stephen J. Wright}{Low-Rank Schwarz Methods for Heterogeneous RTE}

\section{Introduction}

The radiative transfer equation (RTE) is a standard model that
describes propagation of light through such turbid media as biological
tissues or planetary atmospheres. The equation is used in situations
in which energy is transported by light, as in the study of the
greenhouse effect \cite{barcza2017greenhouse}, optical tomography
\cite{KLOSE2002691}, and the radiation field for atmosphere-ocean
system \cite{Plass:69}. Light is injected from a source, and RTE
models the absorption and scattering of the photons in the ambient
material.

The model equation for the steady state is
\begin{equation}\label{eqn:RTE_plain}
	v\cdot \nabla_x u(x,v) = \sigma(x) \Lcal u(x,v)\,, \quad (x,v)\in \Dcal :=\Kcal \times \Vcal,
\end{equation}
where $u(x,v)$ describes the light intensity at location $x$ oriented
in velocity direction $v$. The  left-hand side describes free
propagation of the photons along direction $x$ with velocity $v$,
while the  right-hand side characterizes interaction between photons
and media (via absorption and scattering). The media information is
encoded in $\sigma(x)$, which is strictly positive for all $x$. The
operator $\Lcal$, typically an integral operator, characterizes how
photons are scattered and change directions. We denote the physical domain by $\Kcal$.  Since photons always move
with the same speed, the velocity term is determined purely by the
direction, so that $v\in \Vcal = \mathbb{S}^{d-1}$, the unit sphere in
$d$ dimensions.

In the large space-regime, with scaling $x\to\frac{x}{\vep}$ (where
$\vep$ is a small parameter discussed below), the equation
\eqref{eqn:RTE_plain} becomes
\begin{equation}\label{eqn:RTE}
	\vep v\cdot \nabla_x u(x,v) = \sigma^\delta(x) \Lcal u(x,v)\,,
\end{equation}
where $\sigma^\delta(x)$ is the rescaled media function, with $\delta$
capturing the smallest scale of the variation in the media.  This
function is rough when $\delta \ll 1$. In the equation \eqref{eqn:RTE}, $\vep$ is the Knudsen number that represents the ratio of
the mean free path to the typical domain length.

With appropriate boundary conditions, well-posedness of the equation
is straightforward, and is independent of the scales (that is, the
smallness of $\vep$ or $\delta$)~\cite{agoshkov2012boundary,Lions}. 
In this paper, we tackle the numerical challenge of
  designing an efficient numerical solver for \eqref{eqn:RTE}. We are
  especially interested in the regime of small $\vep$ and small
  $\delta$, where classical numerical methods typically require high
  memory and computational cost, as we explain below.

\subsection{Asymptotic preserving}\label{sec:AP}

Small parameters in PDEs can induce computational challenges. In the case described above, we have
$\nabla_xu \sim {1}/{\min\{\vep\,,\delta\}}$, so a classical numerical
solver can be expected to attain good accuracy only when the mesh size
in the discretization $\Delta x$ satisfies
\[
\Delta x\ll \min\{\vep\,,\delta\}\,.
\]
A grid in $d$ dimensions with this discretization parameter will have
at least $N \gg \min (\vep,\delta)^{-d}$ grid points, so the computation is prohibitive when $\vep$ and $\delta$ are small.

A natural question is whether it is possible to design a numerical
method for which the computational cost of obtaining a stable,
accurate solution is independent of the parameters $\vep$ and
$\delta$, and whether the numerical solution can capture the right
asymptotic limit of the solution as $\vep$ and $\delta$ approach zero.
If a numerical solver for a multiscale problem has its discretization
independent of the smallest scale in the equation, but still preserves
the asymptotic limits, then the solver is called asymptotic-preserving
(AP). This term was coined in~\cite{jin1999efficient} for a class of
kinetic equations, although some algorithms for simpler settings had
been designed previously \cite{LARSEN1987283}. Extensive progress has
been made during the past decade, with AP solvers being designed for
the Bhatnagar-Gross-Krook equation (BGK, a special simplified version
of the Boltzmann model that keeps the equilibrium), the Boltzmann
equation, the Vlasov-Poisson-Boltzmann (VPB) equation, and many
others~\cite{lemou2008new,crouseilles:hal-00533327,dimarco_pareschi_2014,degond2011asymptotic,HU2017103}. See
also the reviews~\cite{jin2010asymptotic}.

A standard approach for designing AP solvers is based on analysis of
the asymptotic limits.  In some cases, asymptotic limits for the
equations can be derived: the Euler limit for the Boltzmann, the
coupled diffusion-Poisson system for the VPB system. One strategy for
obtaining the AP property is to work with two sets of solvers, one for the original
equation and one for the asymptotic limit, the latter being encoded in
the former via a weight that can be tuned. In the limit as
$\vep\to 0$, this weight is adjusted so that the limiting equation solver
dominates, driving the numerical solution to that of the asymptotic
limiting system.

The analysis-based approach is straightforward and mathematically
sound, and has made some previously impossible computations
feasible. It depends, however, on analytical understanding of the
asymptotic limit, which is not always straightforward. We are led to
ask whether it is possible to design an AP solver that does not
require detailed knowledge of the asymptotic limit.  This paper
addresses this question in the specific case of RTE. This equation is
complicated in that different patterns of convergence of the pair
$(\vep,\delta)$ to $(0,0)$ lead to different limiting systems, not all
of which are well understood. Can we design AP numerical solvers in
the absence of this analytical understanding?  We outline an answer to
this question in the next section.

\subsection{Random Sampling and PDE Compression} 
We design AP solvers without analytical knowledge by using {\em compression} techniques.  Even when asymptotic limits of equations
with small parameters are difficult to derive analytically, we can
sometimes show the {\em existence} of such limits. In the discrete
setting, a basic second order discretization scheme with $\Delta x = o(\sqrt{\tau}\epsilon,\sqrt{\tau}\delta)$ should suffice to attain the accuracy level $\tau$, which leads to $N_\vep =
\min\{{\vep\sqrt{\tau}},\delta\sqrt{\tau}\}^{-d}$ grid points for $O(1)$ domain-size. When a limiting equation exists, this accuracy
level may be attainable with as few as $N = \tau^{-d}$ grid
points. Since the limiting equation is asymptotically close to the
original equation, these $N$ degrees of freedom are asymptotically sufficient to
represent the original PDE solution that naively would require
$N_\vep$ grid points to compute. This observation implies that the
$N_\vep$-dimensional solution space is compressible, and can be well
approximated by a $N$-dimensional space when $\vep$ and $\delta$ are
small.

Knowing that the space is ``compressible'', can one find the
compressed space quickly? We answer this question affirmatively, in
the case of the RTE, by making use of random sampling.  Random
sampling is not a new strategy. It has been used in data science to
sample sparse vectors (as in compressed sensing \cite{CanRT06d}) and
low-rank matrices~\cite{halko2011finding}, with the goal of
reconstructing these objects from a relatively small number of
samples. Generally, the number of samples is tied more closely to the
intrinsic dimension of the object (for example, the number of non-zeros
in a sparse vector) than to the dimension of the ambient space, which
is typically much larger.

Applications of random sampling techniques to PDEs have been limited
previously to the discrete algebraic systems obtained by discretizing the
PDEs. A direct link to the original PDEs needs to be explored further.
Some important questions have not been fully resolved, for example,
whether the PDE {\em solution space} or the {\em solution operator} is
``compressible''.  The two views correspond to regarding a matrix as
defining a column space or as a linear operator, respectively. Other
questions involve the sense in which these objects are
``compressible'', and whether the spectral norm used for matrices is
the appropriate norm in the case of PDEs.

Previously, mostly in the context of elliptic PDEs,
  the low rank property of the solution space has been investigated
  and utilized in numerical solvers. For example, homogenization
  theory has been utilized~\cite{HOU1997169,ming2005analysis} for
  designing local basis functions for multiscale problems with
  structured media. For more general setting of $L_\infty$ media, the
  Kolmogorov $N$-width or the problem was studied in a pioneering
  paper~\cite{babuska2011optimal}, while the structure fo the Green's
  functions were investigated in the framework of hierarchical
  matrix~\cite{bebendorf2008hierarchical,hackbusch2015hierarchical}.
  Inspired by the studies, many algorithms have been proposed to
  utilize the rank (or decay) property,
  including~\cite{chung2016sparse,LIPTON201638,Peterseim_mathcomp}. Algorithms
  that specifically use PDE compression and random sampling ideas are
  developed
  in~\cite{calo2016randomized,xia2013randomized,Martinsson11,
    Smetana,DOOSTAN20113015}. In the transport equation
  setting,~\cite{ChungEfendievLiLi_2020} incorporated the random
  sampling technique within the discontinuous Galerkin framework for
  building local solution dictionaries. Corresponding to the first
  question asked above, in most of these papers, the authors regard
  the solution space to be ``compressible'' and the associated matrix
  is regarded as a column space.  A more systematic investigation of
PDE compression appears recently in our previous
work~\cite{chen2018random}, where compressed PDE solution spaces are
related to low rank structure of the matrix formed by the Green's
functions. Such concepts in multi-scale PDE computation as
asymptotic-preserving (see above) and numerical homogenization are
unified under this framework.

\subsection{Contribution}
This paper follows the line of research started in
\cite{chen2018random}. For RTE~\eqref{eqn:RTE}, we know only that the
equation has asymptotic limits with small parameters, but the actual
forms of the limiting equations are unknown. 
We aim to design an accurate
numerical scheme whose runtime is independent of the smallness of the
coefficients in the equation. 

We apply the Schwarz iteration under the domain decomposition
framework.  The domain is divided into overlapping subdomains
(patches). The PDEs in these patches can be solved in
parallel. Solution of the PDE on each patch with partial boundary
conditions yields an output in the form of boundary conditions that
are passed to neighboring patches. The PDE on each patch is solved
again with the modified boundary conditions supplied by its neighbors,
the whole process repeating until the solutions are consistent in the
overlapping regions. The {\em boundary-to-boundary map}, in which the
inputs are the partial boundary conditions on the patch PDEs and the
output are the missing boundary conditions obtained by solving the
PDEs, is a {\em compressible map}.  We will develop an algorithm based
on random sampling that computes an adequate approximation to this map quickly. The overall scheme
is a composition of an offline component, in which low-rank
approximations to the boundary-to-boundary maps are obtained using
random sampling; and an online step, in which Schwarz iteration,
accelerated by the low-rank boundary-to-boundary map, is executed
until a solution consistent across the whole domain is found.

Our work contrasts with the approach in~\cite{chen2018random}, where
the local solution space is compressed in an offline step. In the
online step, a solution for particular boundary conditions or source
term is found as a linear combination of basis vectors for the
compressed space, with the coefficients chosen to match the given
conditions.  The problem of finding these coefficients is typically
overdetermined, the number of coefficients being fewer than the
constraints arising from the boundary conditions or source term.  Some
accuracy is sacrificed, and the error is difficult to quantify.  The
current work compresses the boundary-to-boundary map, rather than the
local solution space, in the offline stage, and uses the compressed
map to update local boundary conditions in the online stage, until a
preset error tolerance is achieved.

In this work, to demonstrate our numerical scheme and
  validate our theory, we consider a simpler setting of $1+1$ problem,
  that is, one spatial dimension and one velocity dimension. Our
  theory can be extended to higher dimensions in a conceptually
  straightforward way, but the implementation of the numerical scheme
  would become significantly more delicate in such cases. We do not
  pursue high-dimensional versions in this paper.

The remainder of the paper is organized as follows.  We introduce the
concept of ``low-rankness'' in the context of the RTE in
Section~\ref{sec:LowRank}. In Section~\ref{sec:method}, we review the
Schwarz iteration under the domain decomposition framework, and
present the new low-rank Schwarz iteration method based on random
sampling. Numerical experience is described in
Section~\ref{sec:Numerical}.

\section{Low-rankness of RTE in Various Regimes}
\label{sec:LowRank}

As discussed above, current AP schemes rely heavily on good
understanding of the analytical form of the asymptotic limits,
although in some situations, this limiting form is hard to specify,
even when we know that it exists. The radiative transfer equation with
small Knudsen number $\vep$ and small media oscillation period
$\delta$ is a good example of the latter phenomenon. As $\vep$ and
$\delta$ converge to $(0,0)$ in different ways, the limiting equations
are different, and only some of the limiting forms can be expressed
explicitly. We show two different homogenization effects in the
following two subsections, and unify them using the concept of the
low-rankness in Section~\ref{sec:low_rank}.

Consider the RTE in infinite domain \eqref{eqn:RTE}, which we restate here:
\begin{equation}\label{eqn:RTE_scaled}
\vep v\cdot \nabla_x u(x,v) =\sigma^\delta(x)\Lcal u(x,v),
\end{equation}
where $x\in\mathbb{R}^d$ and $v\in\Vcal=\mathbb{S}^{d-1}$. We define the scattering operator $\Lcal $ to have the
following form:
\begin{equation*}
\Lcal u(x,v) = \int_{\mathbb{S}^{d-1}} u(x,v') \, \rd \mu(v') - u(x,v)\,,
\end{equation*}
where $\mu(v)$ is the normalized measure on the velocity domain. The
scattering coefficient $\sigma^\delta(x)>0$ encodes the media
information, with $\delta$ denoting the smallest spatial scale.  The
operator has a nontrivial null space $\Null \mathcal{L}$ which
consists of functions that are constant in the velocity domain.  We
use this fact later to formally derive the diffusion limit of
RTE. 
In this article, we choose the operator $\Lcal$
  to have this specific form, for simplicity. In practice, especially
  in applications to atmosphere science, the radiative transfer
  equation often takes this operator to be $\Lcal u(x,v) =
  \int_{\mathbb{S}^{d-1}}k(x,v,v')u(x,v')\,\rd\mu(v')
  -\sigma(x,v)u(x,v)$ with the collision kernel
	\[
	k(x,v,v')=\frac{1-g^2}{4\pi(1+g^2 - 2g v\cdot v')^{3/2}}\,,
	\] 
	where the constant $g\in[-1,1]$ determines the relative strength
    of the forward and backward scattering.  This is the so-called
    Henyey-Greenstein model. The equation would have the same type of
    asymptotic limit (an elliptic equation) as long as
    $\sigma(x,v) = \int
    k(x,v,v')\rd\mu(v')$~\cite{agoshkov2012boundary,Lions}.
  


We now consider different limits for different regimes of the
parameters $(\vep,\delta)$.

\subsection{Diffusion Regime}

In the diffusion regime, we have  $\vep\to 0$ while $\delta$ is
fixed at a positive value. From~\eqref{eqn:RTE_scaled}, we see that 
$\mathcal{L}u(x,v) \sim 0$ in the leading order, meaning that $u(x,v)$
belongs to the null space of $\mathcal{L}$, and loses its velocity
dependence. By matching orders in the  classical asymptotic expansion
\[
u(x,v) = u_0(x,v)+\vep u_1(x,v) +\cdots\,,
\]
we obtain
\begin{alignat*}{2}
\mathcal{O}(1): & \quad u_0(x,v)\in\text{Null}\mathcal{L}\,,\quad && u_0(x,v) = u_0(x)\,,\\
\mathcal{O}(\vep):& \quad v\cdot \nabla_x u_0(x,v)=\sigma^\delta\mathcal{L}u_1\,,\quad &&u_1(x,v) = -\frac{1}{\sigma^\delta}v\cdot\nabla_xu_0(x)\,,\\
\mathcal{O}(\vep^2):& \quad v\cdot \nabla_x u_1(x,v)=\sigma^\delta\mathcal{L}u_2\,,\quad &&\smallint v\cdot \nabla_x u_1(x,v)\rd\mu{(v)} = 0\,.
\end{alignat*}
By substituting the $\mathcal{O}(\vep)$ equation into the
$\mathcal{O}(\vep^2)$ equation, one obtains a diffusion equation, as
follows.
\begin{theorem}[\cite{bardos1984diffusion}]
In the zero limit of $\vep$, the solution to~\eqref{eqn:RTE_scaled}
converges to the solution to the diffusion equation:
\begin{equation}\label{eqn:Diff}
\nabla_x \cdot \left(\frac{1}{\sigma^\delta(x)} \nabla_x u_0(x) \right) = 0\,,
\end{equation}
in the sense that
\[
\|u(x,v)-u_0(x)\|_{L_2(\rd x \, \rd\mu(v))} = \mathcal{O}(\vep)\,.
\]
\end{theorem}

\begin{remark}\label{rmk:boudary}
Note that we did not account for boundary conditions in deriving the
limiting equation. In physical space, the derivation is valid when the
boundary conditions are periodic. Otherwise, one has to be careful
with the boundary influences and curvature effects.
The diffusion limit still holds outside the boundary layers, but the
convergence deteriorates when curvature corrections need to be taken
into account. These results can be found in \cite{Guo2017,Li2017} for
the case when domain is convex.
\end{remark}


\subsection{Homogenization Regime} 

When $\vep$ is fixed at a positive value while $\delta\to 0$,
homogenization limits are achieved;
see~\cite{dumas2000homogenization}.  We assume a two-scale media,
having dependence on a fast variable $y = \frac{x}{\delta}$ and a slow
variable $x$:
\[
\sigma^\delta(x) = \sigma \left(x,\frac{x}{\delta} \right),
\]
where $\sigma(x, \cdot )$ is assumed to be periodic (with respect to the
fast variable) for each $x$. Accordingly, we write the solution as
\[
u^\delta(x,v) = u(x,y,v) = u \left(x,\frac{x}{\delta},v \right).
\]
In this notation, the operator $\nabla_x$ is replaced by
$\nabla_x+\frac{1}{\delta}\nabla_y$ from chain rule. By substituting
into the equation, 
we have
\[
v\cdot \nabla_x u^\delta(x,y,v) + \frac{1}{\delta} v\cdot \nabla_y u^\delta(x,y,v) = \frac{\sigma(x,y)}{\vep}\mathcal{L}u^\delta\,.
\]
By substituting the asymptotic expansion
\[
u^\delta(x,y,v) = u_0(x,y,v) +\delta u_1(x,y,v) + \mathcal{O}(\delta^2)\,,
\]
into the equation above, and matching terms, we obtain
\begin{subequations}
  \begin{alignat}{2}
    \label{eq:su7}
    \mathcal{O}(1/\delta):& \quad  & v\cdot\nabla_yu_0(x,y,v)&=0\,,\\
    \label{eq:su8}
\mathcal{O}(1):& \quad  & v\cdot \nabla_x u_0-\frac{\sigma(x,y)}{\vep}\mathcal{L}u_0 &= v\cdot\nabla_yu_1(x,y,v)\,.
\end{alignat}
\end{subequations}
By applying the Fourier transform for the first equation
\eqref{eq:su7} with respect to the periodic variable $y$, we obtain
\[
i2\pi v\cdot \xi \hat{u}_0(x,\xi,v) = 0\,, \quad \mbox{for all  $\xi \in \mathbb{Z}^d$.}
\]
We note that for almost all fixed $v\in \mathbb{R}^d$, the multiplier
$i2\pi v\cdot \xi$ is non-vanishing for all $\xi\in \mathbb{Z}^d
\backslash\{0\}$, because otherwise there exists some $\xi\in
\mathbb{Z}^d\backslash \{0\}$ such that $i2\pi v\cdot \xi = 0 $ for a
positive measure set of $v$, which is impossible. Therefore, by
dividing the multiplier, we have for any $\xi\in
\mathbb{Z}^d\backslash \{0\}$ that
\[
\hat{u}_0(x,\xi,v) = 0\,, \quad \text{for almost all } (x,v)\,.
\]
That is, all Fourier modes are vanishing except for the one with $\xi
= 0$. This implies that $u_0$ is independent of the periodic variable
$y$, so we redefine the notation to omit this dependence:
\[
\quad u_0(x,y,v) = u_0(x,v) \,, \quad \text{for almost all } (x,v)\,.
\]
For the next order equation \eqref{eq:su8}, when we take the integral
over $y$, the RHS vanishes due to periodicity, and we obtain the
homogenized equation
\[
v\cdot \nabla_x u_0(x,v) = \frac{\sigma^\ast(x)}{\vep} \Lcal
u_0(x,v)\,, \quad\text{with}\quad \sigma^\ast = \smallint
\sigma(x,y)\rd{y}\,.
\]

The  derivation of homogenization limit presented above is validated
 in the following theorem. See
\cite[Theorem~3.1]{dumas2000homogenization} for a rigorous proof for
the time-dependent case.
\begin{theorem}
	Let $\sigma^\delta(x)$ be a bounded family of $L^\infty$
        functions such that
	\begin{equation*}
      \sigma^\delta  \xrightarrow{\delta\rightarrow 0} \sigma^\ast \,, \quad \text{in} \  L^\infty\ \text{weak-}\ast \,\text{topology},
	\end{equation*}
	then the solution $u^\delta(x,v)$ to the RTE
        \eqref{eqn:RTE_scaled} (with $\vep$ fixed) converges in
        $L^\infty$ weak-$\ast$ topology to $u(x,v)$, the solution to
        the following homogenized RTE:
	\begin{equation}\label{eqn:RTE_hom}
	\vep v\cdot \nabla_x u(x,v)  = \sigma^\ast(x)\Lcal u(x,v)\,.
	\end{equation}
\end{theorem}

Because of the oscillations of scale $\delta$ in the media
$\sigma^\delta(x)$, the solution $u^\delta$ is rough. However, such
oscillations are homogenized in the $\delta\to0$ limit, and the
solution $u^\delta$ becomes close to the solution
to~\eqref{eqn:RTE_hom}, which has no oscillation.

In general, the limiting regime $\vep \rightarrow 0,\delta \rightarrow
0$ can be taken through different routes. One may fix $\delta=1$ and
send $\vep$ to zero to reach the diffusion limit \eqref{eqn:Diff} and
then send $\delta \rightarrow 0$, shown as solid arrow in
Figure~\ref{fig:diagram}, Alternatively, one could fix $\vep=1$ and
send $\delta$ to zero to reach the homogenization limit
\eqref{eqn:RTE_hom}, then send $\vep \rightarrow 0$. This path is
shown as the dashed arrow in Figure~\ref{fig:diagram}. Additionally,
one could send both $\vep$ and $\delta$ simultaneously to zero at
different rates, shown as dotted arrows in
Figure~\ref{fig:diagram}. All these routes, though considering the
same regime $\vep,\delta \rightarrow 0$, do not necessarily end up at
the same limit. In fact, Goudon and
Mellet~\cite{goudon2001diffusion,goudon2003homogenization} showed that
by following the route $\vep = \delta \rightarrow 0$, RTE
\eqref{eqn:RTE} ends up as an effective drift diffusion equation,
while Abdallah, Puel and Vogelius~\cite{abdallah2012diffusion}
followed the route $\delta \gg\vep\rightarrow 0$ to obtain an
effective diffusion equation.



\begin{figure}[htb]
\centering \includegraphics[width =
  0.9\textwidth]{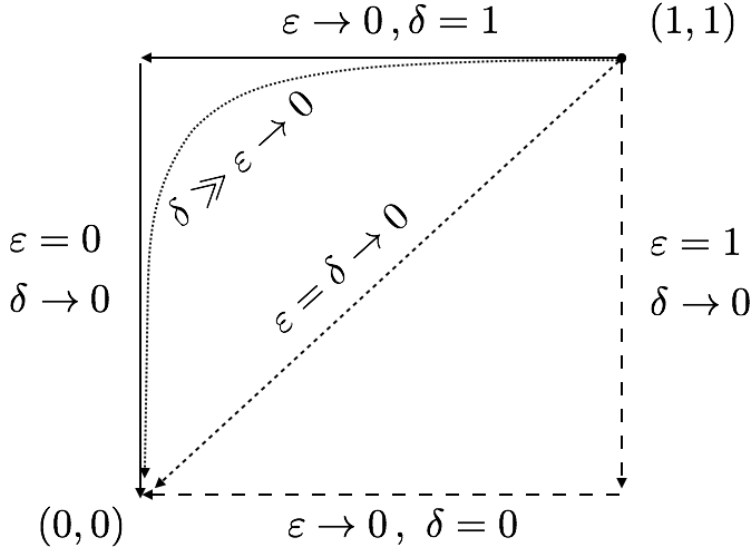}
\caption{
  The diffusion limit theory is established by Bardos, Santos and
  Sentis~\cite{bardos1984diffusion} through the horizontal arrow
  $\varepsilon \rightarrow 0,\delta = 1$. Dumas and
  Golse~\cite{dumas2000homogenization} considered the vertical arrow
  $\vep = 1,\delta \rightarrow 0$. Goudon and
  Mellet~\cite{goudon2001diffusion,goudon2003homogenization}
  considered the diagonal path $\vep=\delta \rightarrow 0$, while
  Abdallah, Puel and Vogelius~\cite{abdallah2012diffusion} studied the
  curved path $\delta \gg \vep \rightarrow0$. Different limiting
  equations might arise for different regimes; the diagram does not
  commute.
}
\label{fig:diagram}
\end{figure}

\subsection{Low Rank of the PDE Solution Map}\label{sec:low_rank}
An AP scheme was proposed in~\cite{qinMMs} to deal with the regime
$\delta\gg \vep \rightarrow 0$, while numerical schemes for other
regimes remain open. In practice, given a particular pair
$(\vep,\delta)$ that is close to zero, it is impossible to determine {\em which} limiting equation is the most
appropriate one to use as an approximation to the solution of
\eqref{eqn:RTE}. The analysis-based approach of designing AP schemes
is therefore not feasible. We seek to develop instead a universal {\em
  numerical} approach that is valid in different limiting regimes.

We start by considering the diagram in
Figure~\ref{figure:low-rank}.
\begin{figure}[htb]
\centering
\includegraphics[width = 0.9\textwidth]{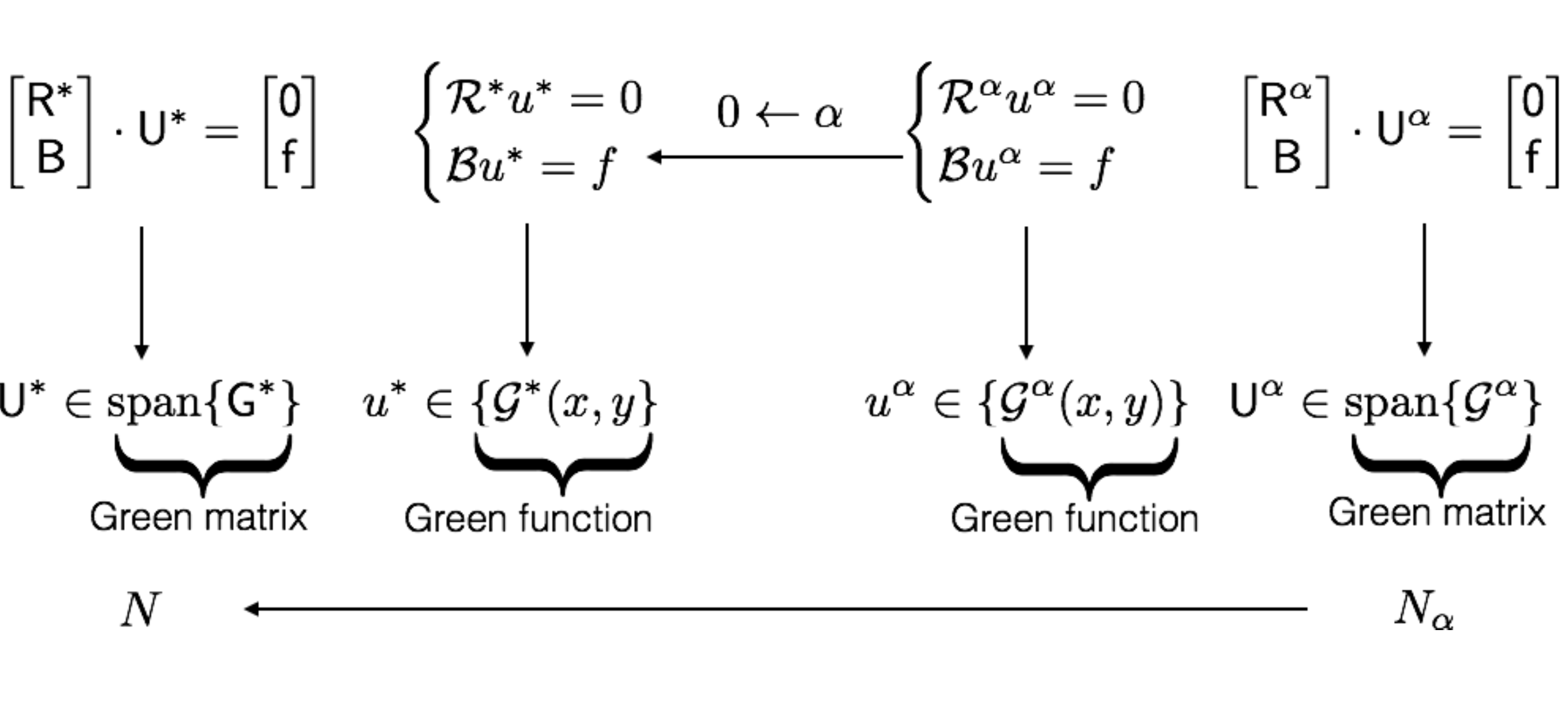}
\caption{Low rankness of systems with small parameters.}
\label{figure:low-rank}
\end{figure}
Assume we are given an equation $\mathcal{R}^\alpha u^\alpha = 0$
where $\alpha:=\min\{\vep,\delta\}$ denotes the smallest parameters in
the equation operator $\mathcal{R}^\alpha$, together with a boundary
operator $\mathcal{B}$ such that $\mathcal{B} u^\alpha = f$ for some
given boundary data $f$. The solution can be represented as a
convolution of $f$ with all Green's functions $\mathcal{G}^\alpha$, so
it lies in the space spanned by $\mathcal{G}^\alpha$. To find an
accurate numerical approximation to this solution, the operator
$\mathcal{R}^\alpha$ is translated to $\mathsf{R}^\alpha$, a matrix
with $N_\alpha \ge \frac{1}{\alpha}$ columns. Thus, the numerical
solution $\mathsf{U}^\alpha$ is a vector of length $N_\alpha$.

On the other hand, assume the equation is ``homogenizable'' and there
exists an asymptotic limit, an operator $\mathcal{R}^\ast$ so that the
solution $u^\ast$ to equation $\mathcal{R}^\ast u^\ast=0$ with
boundary condition $\mathcal{B} u^\ast = f$ is asymptotically
close to $u^\alpha$. The computation of $u^\ast$ is expected to be
significantly cheaper, since the limiting equation no longer has small
parameters and is expected to be smooth. Thus, the numerical solution
$\mathsf{U}^\ast$ requires merely $N=\mathcal{O}(1)$ degrees of freedom to
represent $u^\ast$ accurately.

Since $u^\ast$ is close to $u^\alpha$ for regular $f$,
the numerical solution spaces captured by the range of matrices
$\mathsf{G}^\ast$ and $\mathsf{G}^\alpha$ are expected to be almost the
same. On the other hand, although $\mathsf{G}^\alpha$ contains many more
degrees of freedom (columns) than $\mathsf{G}^\ast$, the former is
``compressible'' and the latter is a good low-rank approximation to
it. This argument can be made rigorous with the definition of
``numerical rank of an operator,'' a concept that is
equivalent to the ``Kolmogorov $n$-width''. More details can be found
in \cite{chen2018random}.

\subsection{Random Sampling for Low-Rank Structure}

Knowing that an operator is approximately of low rank does not mean
that it is easy to find a low-rank approximation quickly. In the
linear algebra setting, finding the low-rank structure is equivalent
to finding the singular vectors of a matrix that correspond to the
largest singular values. 
For an $n\times m$ matrix, the singular value
  decomposition costs $\mathcal{O}(nm \min\{n, m\})$ operations,
  making the computation expensive for large matrices.  When the
approximate rank is known to be $r \ll n$, a randomized SVD (RSVD)
solver based on a sketching procedure is available, whose cost depends
on $r$. Properties of this approach are described in the following
result.
\begin{theorem}[Corollary~10.9 of \cite{halko2011finding}]
Suppose that the matrix $\mathsf{A}\in\mathbb{R}^{n \times m}$ has
singular values ordered as follows: $\sigma_1\geq \sigma_2\geq
\cdots$.  Assume that the target rank $r$ and oversampling parameter
$p\geq 4$ are positive integers such that $k:=r+p \le
\min\{m,n\}$. Then with probability at least $1-6e^{-p}$, we have
\[
\|\mathsf{A}-\mathsf{Q}\mathsf{Q}^\top\mathsf{A}\|_2\leq \left(
1+17\sqrt{1+r/p} \right)\sigma_{r+1} +
\frac{8\sqrt{k}}{p+1}\Bigl(\sum_{j>r}\sigma_j^2\Bigr)^{1/2} \,,
\]
where $\mathsf{Q} \in \mathbb{R}^{m\times k}$ is a matrix with
orthonormal columns whose column space matches that of
$\mathsf{A}\Omega$, where $\Omega$ is a random matrix of dimension $m
\times k$ with entries drawn from an independent identical distributed
(i.i.d.)  normal distribution.
\end{theorem}

This theorem suggests that if an operator has approximate low rank,
random sampling can find its range accurately, with overwhelming
probability. A simplified estimate shows that for oversampling
parameter $p$ as small as $5$, with at least $99.8\%$ confidence, the
error $\| \Amat - \Qmat\Qmat^\top \Amat \|_2$ can be controlled by
$\left( 1+ 11\sqrt{ \min\{m,n\}(r+5)}\right)\sigma_{r+1}$.
Algorithm~\ref{alg:RSVD}, proposed in \cite{halko2011finding}, finds
the rank-$r$ approximation to $\Amat$, which we denote by $\Amat_r$.
\begin{algorithm}[h]
	\caption{Randomized Singular Value Decomposition (RSVD)}\label{alg:RSVD}
	\begin{algorithmic}[1]
	  \State Given matrix $\Amat\in\mathbb{R}^{n\times m}$, target rank $r$ and oversampling parameter $p$;
         \State Set $k := r+p$;
		\State \textbf{Stage I:}
		\Indent
		\State Generate matrix $\Omega\in\mathbb{R}^{m\times k}$ with i.i.d. normal variables and compute $\Ymat = \Amat \Omega \in \R^{n \times k}$;
		\State Perform QR-decomposition and obtain $\left[\Qmat,\Rmat\right] = qr(\Ymat,0)$, where $\Qmat \in \R^{n \times k}$ has orthonormal columns;
		\EndIndent
		\State \textbf{Stage II:}
		\Indent
		\State Form $\Bmat = \Qmat^\ast \Amat \in \mathbb{R}^{k \times m}$;
		\State Compute SVD $\tilde{\Umat}\Sigmamat\Vmat^\ast$ of $\Bmat$, where $\tilde{U} \in \R^{k\times k}$ and $V \in \R^{m \times k}$ are matrices with orthonormal columns, and $\Sigma \in \R^{k \times k}$ is a diagonal matrix with nonnegative diagonals;
		\State Compute $\Umat = \Qmat\tilde{\Umat} \in \R^{n \times k}$, noting that $\Umat$ has orthonormal columns;
		\EndIndent
		\State \textbf{Return:} $\Amat_r = \sum_{i=1}^r \Umat_{\cdot i} \Sigmamat_{ii} \Vmat_{\cdot i}^T$, where $\Umat_{\cdot i}$ and $\Vmat_{\cdot i}$ denotes column $i$ of $\Umat$ and $\Vmat$, respectively.
	\end{algorithmic}
\end{algorithm}

There are two crucial features of the algorithm. First, the amount of
computation depends crucially on the rank $r$, and is generally much
less expensive than a full SVD. Second, it can be implemented without
explicit knowledge of the matrix $\Amat$. Rather, we need only to be
able to compute the products of $\Amat$ with the random matrix
$\Omega$.  These properties make the algorithm well suited for use in
the numerical homogenization of PDEs.

RSVD is not the only algorithm that achieves the
  decomposition at the cost of $O(r)$ matrix-vector
  multiplications. Another approach is to explore a Krylov subspace of
  rank (slightly) greater than $r$ by initializing with a random
  vector and multiplying repeatedly by $\Amat$. The first $r$ singular
  values of the resulting matrix can be taken as approximating the
  leading $r$ singular values of $\Amat$. A variation of RSVD incorporates power
  iteration, which is similar to constructing a Krylov subspace;
  see~\cite[Sections~4.5 and 10.4]{halko2011finding}. We restrict here
  to the original RSVD algorithm for its simplicity and effectiveness
  in our numerical tests.

Translation of the randomization idea to the PDE setting is not
straightforward. First, a PDE solution map is a continuous operator,
not a matrix. Discretization of the space and the choice of norm is
not always obvious. Redesigning the scheme to deal with an operator
may also be difficult.  While the adjoint of a matrix is easy to
define, the adjoint of an operator may not be so easy to
define. Second, a numerical homogenization scheme needs to find a
solution quickly for an arbitrary boundary term or source, and in
pursuit of that goal, the PDE solution map may not be the right
operator to ``compress''. In our approach, discussed in
Section~\ref{sec:method}, we compress a different operator: the
boundary-to-boundary map, 
Schwarz iteration scheme.

For the present, given an operator $\mathcal{A}:
  \mathcal{X}\to \mathcal{Y}$, we assume that the adjoint
  $\mathcal{A}^\ast$ is known. We also assume $\mathcal{X}$ and $\mathcal{Y}$ are finite dimensional and there is an inner
    product structure on $\mathcal{X}$ which allows us to efficiently
    draw random samples. We can then translate
  Algorithm~\ref{alg:RSVD} to the operator setting to find the
  corresponding Kolmogorov $r$-width operator $\mathcal{A}_r$ in
  Algorithm~\ref{alg:operator_RSVD}. (Note that upon fine discretization to meet the preset precision threshold, $\mathcal{X}$ can always be made finite dimensional.)
\begin{algorithm}[h]
	\caption{Randomized Operator Rank Capture}\label{alg:operator_RSVD}
	\begin{algorithmic}[1]
		\State Given an operator $\mathcal{A}:\;\mathcal{X}\to\mathcal{Y}$, where $\mathcal{X}$ and $\mathcal{Y}$ are finite dimensional function spaces. Define target rank $r$ and oversampling parameter $p$, and set $k:=r+p$;
		\State \textbf{Stage I:}
		\Indent
		\State Generate $k$ samples $\omega_1,\dotsc,\omega_{k} \in \mathcal{X}$ and calculate  $\{\mathcal{A}\omega_1\,,\cdots \mathcal{A}\omega_{k}\}$;
		\State Perform Gram-Schmit orthogonalization to obtain $\{q_1\,,\dotsc, q_{k}\}$;
		\EndIndent
		\State \textbf{Stage II:}
		\Indent
		\State Act $\mathcal{A}^\ast$ on $\{q_i\}$ to obtain $\{\mathcal{A}^\ast q_1\,,\cdots,\mathcal{A}^\ast q_{k}\}$;
                \State Seek  $\tilde{u}_i \in \R^{k}$, $\sigma_i \in \R$, and $v_i \in \mathcal{X}$, $i=1,2,\dotsc,k$  such that  $\sum_{i=1}^{k} \tilde{u}_i \sigma_i v_i = (\mathcal{A}^\ast q_1\,,\cdots,\mathcal{A}^\ast q_{k})^T$;
                \State Denoting $\tilde{U} = [ \tilde{u}_1 , \dotsc, \tilde{u}_{k} ] \in \R^{k \times k}$, define $u_j = \sum_{i=1}^{k} q_i \tilde{U}_{ij}$, $j=1,2,\dotsc,k$;
		\EndIndent
		\State \textbf{Return:} $\mathcal{A}_r = \sum_{i=1}^r u_i \sigma_i v_i$.
	\end{algorithmic}
\end{algorithm}
\begin{remark}
	Both Algorithm \ref{alg:RSVD} and its operator counterpart Algorithm~\ref{alg:operator_RSVD} require an input target rank $r$. Such parameter could be obtained from a priori estimate for elliptic type equations~\cite{babuska2011optimal}. However, such guidance is unfortunately absent in the transport equation case. To address this issue, one may consider an adaptive randomized range finder (Algorithm 4.2 of \cite{halko2011finding}), in which a tolerance level is preset and the target rank is determined on the fly.
\end{remark}

\section{Low-Rank Schwarz Domain Decomposition Method}\label{sec:method}
We consider the boundary value problem for RTE \eqref{eqn:RTE_scaled} in the following form:
\begin{subequations}\label{eqn:RTE_bd}
  \begin{alignat}{2}
\vep v\cdot\nabla_x u(x,v) & = \sigma^\delta(x)\Lcal u(x,v) \,, \quad  && (x,v)\in \Dcal =\Kcal\times\Vcal\\
u(x,v) & = \phi(x,v) \,, && (x,v)\in \Gamma_-,
\end{alignat}
\end{subequations}
where the partial boundary $\Gamma_-$ is defined by
\begin{equation}\label{eqn:inflow_bd}
	\Gamma_-: = \{(x,v)\in \partial\Kcal \times \Vcal: -n_x\cdot v > 0\}\,.
\end{equation}
Here, $n_x$ is the outer normal vector at location $x\in \partial
\Kcal$, and $v\cdot n_x$ is expected to be negative for all incoming
velocities. Similarly, the outflow coordinates are collected in the
complementary partial boundary $\Gamma_+: = \{(x,v)\in \partial\Kcal
\times \Vcal: n_x\cdot v > 0\}$. Problem \eqref{eqn:RTE_bd} is well
posed, as we show in the Appendix.

We start in Section~\ref{sec:DD} by introducing domain decomposition
and the classical Schwarz method
(Algorithm~\ref{alg:Schwarz0}). Section~\ref{sec:introduce_Pcal}
identifies the operator that needs to be compressed for efficient
implementation of this method, while Section~\ref{sec:adjoint} derives
the adjoint operator. These elements together make it possible to
design the low-rank Schwarz method (Algorithm~\ref{alg:re_Schwarz}),
which is presented in Section~\ref{sec:reduced_schwarz}.

\subsection{Schwarz Domain Decomposition Method}\label{sec:DD}

To solve~\eqref{eqn:RTE_bd}, we first consider an overlapping domain
decomposition of the physical space $\Kcal$,
\[
\Kcal = \bigcup_{m=1}^M \Kcal_m\,,
\]
where $\{ \Kcal_m \}_{m=1,2,\dotsc,M}$ forms an open cover of
$\Kcal$. We assume in the remainder of the discussion that the
subdomains are ordered so that $\Kcal_m$ can overlap only with
$\Kcal_{m-1}$ and $\Kcal_{m+1}$. We decompose the domain $\Dcal = \Kcal\times\Vcal$ accordingly as:
\begin{equation} \label{eqn:dd} \Dcal = \bigcup_{m=1}^M \Dcal_m =
  \bigcup_{m=1}^M (\Kcal_m \times \Vcal)\,.
\end{equation}
We denote by $\Gamma_{m,\pm}$ the outflow and inflow boundaries for
$\Dcal_m$. We define those parts of the subdomains $\Kcal_m$ and
$\Dcal_m$ that do not overlap with their neighbors as follows:
\begin{equation} \label{eq:Dsm}
  \Kcal_m^\text{s} := \Kcal_m \setminus \left(\Kcal_{m-1} \cup \Kcal_{m+1} \right), \quad
  \Dcal^\text{s}_m := \Dcal_m\setminus\left(\Dcal_{m-1}\cup\Dcal_{m+1}\right) =
  \Kcal^\text{s}_m \times \Vcal.
\end{equation}
Since the inflow boundary $\Gamma_{m\pm1,-}$ of neighboring domain is
partially inside $\Dcal_m$, we further define
\begin{equation}\label{def:E}
\Ecal_{m,m-1}:=\Dcal_m \cap \Gamma_{m-1,-}\,,\quad \Ecal_{m,m+1}:=\Dcal_m \cap \Gamma_{m+1,-}\,,
\end{equation}
so that the outflow boundary $\Gamma^\text{s}_{m,+}$ of
$\Dcal_m^\text{s}$ is the union of the domains  above, that is,
\[
\Gamma^\text{s}_{m,+} = \Ecal_{m,m-1}\cup \Ecal_{m,m+1}\,.
\]
One can see that the restriction on $\Gamma^\text{s}_{m,+}$ of the
local solution in the domain $\Dcal_m$, would partially provide the
inflow boundary condition for its neighboring domain
$\Dcal_{m\pm1}$. This fact will be used later to update local
solutions in each iteration of Schwarz method.
Figure~\ref{fig:overlap} illustrates our setup.

\begin{figure}
\centering
\includegraphics[width=0.6\textwidth]{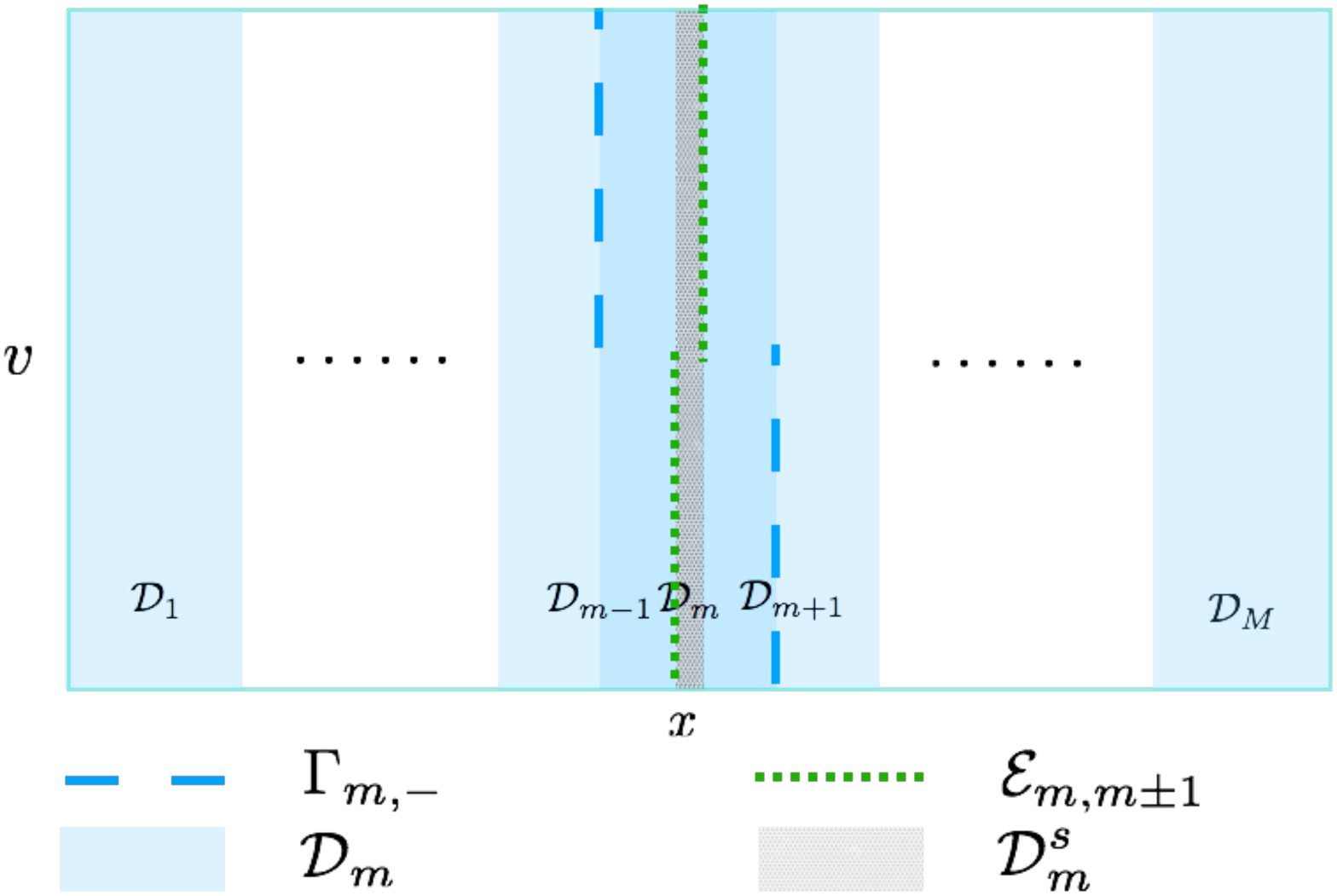}
\caption{An overlapping domain decomposition of $\Dcal$. The
  horizontal direction and vertical direction represent $\Kcal$ and
  $\Vcal$, respectively. For simplicity, we consider only the case in
  which both $\Kcal$ and $\Vcal$ are subsets of the real line. Each
  subdomain $\Dcal_m$ overlaps the neighboring two subdomains
  $\Dcal_{m-1}$ and $\Dcal_{m+1}$ (except that $\Dcal_1$ and
  $\Dcal_{M}$ have one neighboring subdomain). The inflow boundary for
  $\Dcal_m$, $\Gamma_{m,-}$, and the inflow boundary for $\Dcal_{m\pm
    1}$ confined in $\Dcal_m$, which is denoted by $\Ecal_{m,m\pm 1}$,
  is also illustrated.}\label{fig:overlap}
\end{figure}

The Schwarz method is an iterative algorithm that updates solutions
confined to different subdomains by exchanging information between
iterations. In this setting, we update the values on $\Ecal_{m,m\pm1}$
using the newly computed solutions in each subdomain, repeating the
process until the solution converges. To be more specific, we denote
by $\phi^k_m$ the restriction of the solution at the $k$th step of the
Schwarz process on patch $m$, confined to the partial boundary
$\Gamma_{m,-}$, that is,
\[
\phi^k_{m} := u^k|_{\Gamma_{m,-}}\, \quad \phi^k := \{\phi^k_1, \dotsc, \phi^k_M\}.
\]
The $k$th iteration of the Schwarz method can be expressed as a
mapping from the $\phi^k$ to $\phi^{k+1}$, obtained by exchanging the
boundary conditions between adjacent patches and solving the RTE. We
denote this mapping by $\Pcal$, as follows:
\[
\phi^{k+1}=\Pcal (\phi^{k}).
\]
We terminate at an iteration $k$ for which the difference between
$\phi^k$ and $\phi^{k+1}$ falls below a given tolerance.

The evaluation of map $\Pcal$ amounts to evaluation and assembly of
the individual maps $\Pcal_m(\phi_m)$, $m=1,2,\dotsc,M$, and is
defined by the following procedure.

\medskip

\begin{itemize}
\item [\textbf{step 1}] Define the following solution map
\begin{equation}
\begin{aligned}
\Scal_m : \quad &  L^2(\Gamma_{m,-}; |n\cdot v| ) &\rightarrow \quad & H_A(\Dcal_{m}) \\
& \phi &\mapsto \quad & u
\end{aligned}\,.
\end{equation}
by solving RTE on domain $\Dcal_m$:
\begin{equation}\label{eqn:RTE_local}
\begin{cases}
v\cdot \nabla_x u(x,v)  = \frac{1}{\vep} \sigma^\delta(x)\Lcal u(x,v) &\quad \text{in} \quad \Dcal_m \\
u(x,v) = \phi(x,v) &\quad \text{on} \quad \Gamma_{m,-}
\end{cases}\,,
\end{equation}
Obtain the solution $u_m = \Scal_m(\phi_m)$ in each subdomain
$\Dcal_m$ with boundary conditions $\phi_{m}$. Here $H_A(\Dcal_m)$ is
a functional space where the trace of $u$ over the boundary
$\Gamma_{m,\pm}$ is well defined (see the Appendix for details).

\item[\textbf{step 2}] Confine the solution $u_m$ on the boundaries:
\begin{equation}\label{eqn:MidSchwarz1}
\begin{aligned}
\phi_{m-1}^+ = u_m \quad \text{on}\; \Ecal_{m,m-1} \quad\text{and}\quad \phi_{m+1}^+ = u_m \quad \text{on}\; \Ecal_{m,m+1}\,,
\end{aligned}
\end{equation}
where $\phi^+_{m-1}$ and $\phi^+_{m+1}$ are the boundary values
transmitted to the next iteration of the Schwarz procedure.
\end{itemize}

\medskip

$\Pcal_m$ is the boundary-to-boundary map that maps the boundary
condition on $\Gamma_{m,-}$ to the boundary values on the adjacent
subdomains ($\Gamma_{m,+}^s=\Ecal_{m,m+1}\cup\Ecal_{m,m-1}$):
\begin{equation}
\begin{aligned}
\Pcal_m : \quad &  L^2(\Gamma_{m,-}; |n\cdot v| ) &\rightarrow \quad & L^2(\Gamma^\text{s}_{m,+}; |n\cdot v| ) \\
& \phi &\mapsto \quad & u|_{\Gamma_{m,+}^s}
\end{aligned}\,.
\end{equation}
This map $\Pcal_m$ is a well-defined operator, as we show in
Theorem~\ref{thm:wellposed}. It can be regarded as a composition of
$\Scal_m$ and a trace operator, as follows:
\begin{equation}\label{eqn:update}
	\Pcal_m:\;\phi_m \xrightarrow{\Scal_m}u_m\rightarrow u_m|_{\Gamma^\text{s}_{m,+}}\,.
\end{equation}
Note that $u_m|_{\Gamma^\text{s}_{m,+}}$ provides the boundary
condition for the adjacent subdomains $\phi_{m\pm1}$ in the next
Schwarz iteration, seen as in equation~\eqref{eqn:MidSchwarz1}. The full
map $\Pcal$ is obtained by collecting the action of $\Pcal_m$ for all
subdomains $m=1,2,\dotsc,M$.

%
%
%
%
%
%

As initial conditions for the Schwarz process, we set
\begin{equation}\label{eqn:InitSchwarz0}
\phi_{m,-}^0 = 0\,,\quad m=1,2,\dotsc,M,
\end{equation}
except at the physical boundary, where we impose given boundary
conditions: 
\begin{equation}\label{eqn:InitSchwarz1}
\begin{aligned}
\phi_{1,-}^0 = \phi^{\text{bdry}} \quad \text{on} \quad \Gamma_-\cap
\Gamma_{1,-}\,,\quad\text{and}\quad \phi_{M,-}^0 = \phi^{\text{bdry}}
\quad \text{on} \quad \Gamma_-\cap \Gamma_{M,-}\,.
\end{aligned}
\end{equation}
When convergence to a given tolerance is achieved, the latest
solutions may not perfectly match at the overlapping areas. To
assemble the global solution, we define a suitable set of
partition-of-unity functions $\{\eta_m(x)\}$ for the subdomains
$\Kcal_m$, whose properties are as follows:
\begin{align*}
  0<\eta_m(x)\leq 1, \;\;  \eta_m(x)&=0  \;\; \mbox{for $x \notin \Kcal_m$ and all $m=1,2,\dotsc,M$}; \\
\sum_{m=1}^M\eta_m(x) & \equiv 1\,,\quad \mbox{for all  $x\in\Kcal$.}
\end{align*}
We construct the global solution by setting
\begin{equation}\label{eqn:assemble}
	u^\text{final}(x,v) = \sum_{m=1}^M u_m(x,v) \eta_m(x) \,.
\end{equation}
The method is summarized in Algorithm~\ref{alg:Schwarz0}.

\begin{algorithm}
	\caption{Schwarz Method for RTE}\label{alg:Schwarz0}
	\begin{algorithmic}[1]
		\State \textbf{Input:} global boundary conditions $\phi^{\text{bdry}}$ and error tolerance $\tau$;
		\State Set $t \leftarrow 0$; Initialize $\phi_m^0$ from boundary conditions  \eqref{eqn:InitSchwarz0} and  \eqref{eqn:InitSchwarz1};
		\State \textbf{Repeat }
		\Indent
		\State $t \leftarrow t+1$;
		\State	\textbf{For $m = 1,\ldots,M$}
		\Indent
		\State $u^t_m \leftarrow \Scal_m(\phi_m^{t-1})$ via \eqref{eqn:RTE_local};
		\State $\phi_{m\pm1}^{t} \leftarrow u_m^t|_{\Ecal_{m,m\pm 1}}$ via trace restrictions~\eqref{eqn:MidSchwarz1};
		\EndIndent
		\State \textbf{EndFor}
		\State $\text{error} \leftarrow \sum_{m}\|\phi_m^t-\phi_m^{t-1}\|$;
		\EndIndent
		\State \textbf{Until $\text{error} \le \tau$};
		\State Assemble the final solution using~\eqref{eqn:assemble};
		\State \textbf{Return:} final solution $u^{\text{final}}$.
	\end{algorithmic}
\end{algorithm}

The Schwarz approach has several advantages. First, it is easy to
implement in parallel, since the main computations
\eqref{eqn:RTE_local} and \eqref{eqn:MidSchwarz1} can be solved
simultaneously for the subdomains $m=1,2,\dotsc,M$. In fact, one could
even use different solvers in different subdomains, when appropriate
(for example, when there is prior information about inhomogeneity of
the medium). Second, computing solutions in each subdomain is
significantly cheaper than for the full domain. It saves storage cost
and computation time, especially when stoge and computation scale
superlinearly with the size of the domain.

The disadvantage of the Schwarz approach is that it requires multiple
iterations for convergence. Since $\mathcal{P}_m$ needs to be
reevaluated at each iteration for each subdomain $\Dcal_m$, and it
calls for the computation of $\Scal_m$, finding the local solutions
with the given boundary condition quickly is the key to the success of
the entire algorithm. In the following sections, we identify the
operator that can be efficiently compressed, aiming at improving the
efficiency of evaluating $\Pcal_m$, or $\Scal_m$.


\subsection{Identifying the Operator to be Compressed}\label{sec:introduce_Pcal}

As discussed in Section~\ref{sec:low_rank}, one should be able to reveal and exploit the low-rankness in homogenizable equations.
In our setting, the local
  equation~\eqref{eqn:RTE_local} has a homogenization limit when
  $\vep$ and $\delta$ are small, so we expect the map from boundary
  conditions to local interior solutions (upon eliminating a boundary
  layer) to be of low rank. Indeed, we see this phenomenon in
  Figure~\ref{fig:sv_layer}, where we plot all normalized singular
  values of the discrete representation of $\Scal_4$ and $\Pcal_4$. A
  solution with an inhomogeneous boundary condition can have strong
  boundary layer effect. These boundary layer effects are included in
  $\Scal_m$, an operator that maps the boundary condition to the
  solution in the entire region (including the boundary layer),
  destroying the desired low-rank structure. However, the operator
  $\Pcal_m$ looks only at the solution confined to a small interior
  set $\Gamma^\text{s}_{m}$, and has a much faster decay in its
  singular values. This observation resonates with the argument in
  Remark~\ref{rmk:boudary}: the homogenization limit concerns mainly
  the behavior of the solution in the interior of the (sub)domain,
  while the behavior of the solution in boundary layers is usually
  still far from ``equilibrium''.
%
\begin{figure}[t]
\centering
\includegraphics[width = 0.40\textwidth]{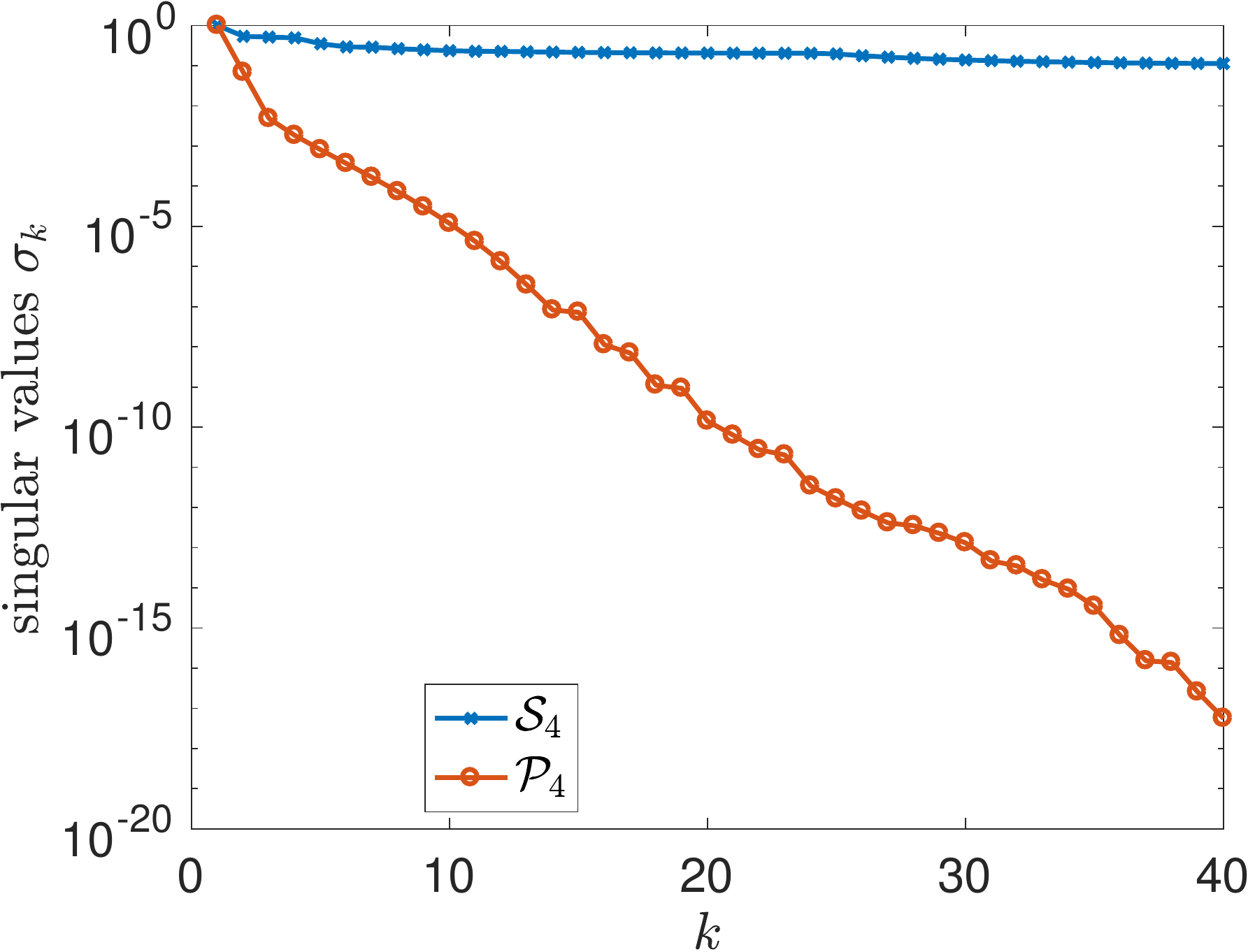}
\includegraphics[width = 0.40\textwidth]{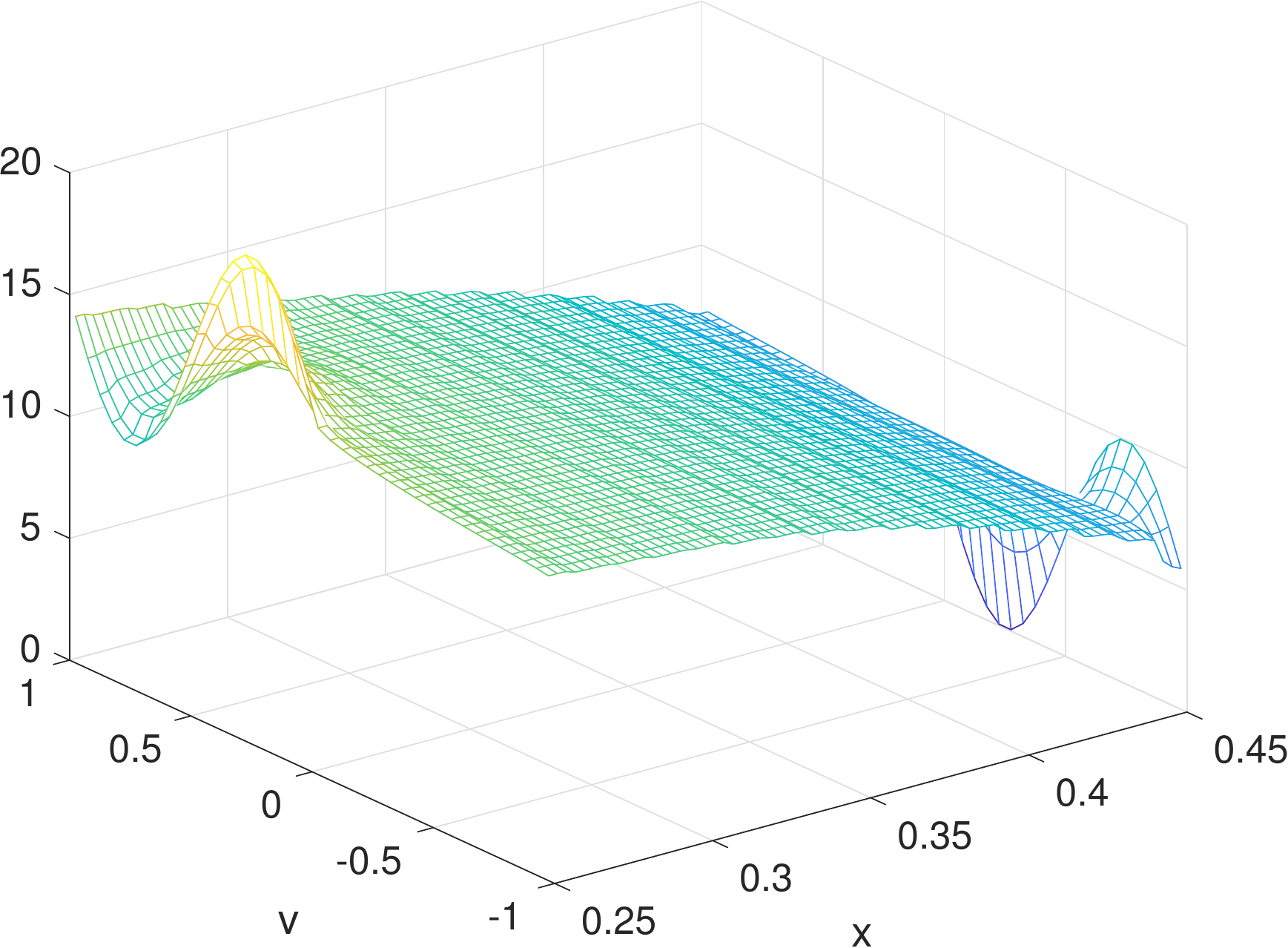}
\caption{Left: Singular values of $\Scal_4$ and $\Pcal_4$. Right: A
  solution with inhomogeneous boundary condition, exhibiting the
  boundary layer effect. $(\vep,\delta) = (1/81,1/81)$ in both
  cases. The domain is chosen to be
    $(x,v)\in[3/20,7/20]\times[-1,1]$ with discretization parameters
    $\Delta x = 0.002$ and $\Delta v = 0.05$. In $x$ direction we
    apply upwind and in $v$ direction we apply the classical $S_n$ (discrete ordinates)
    method. The resulting algebra problem is computed using
    GMRES~\cite{li2017implicit}.}
\label{fig:sv_layer}
\end{figure}

Thus $\Pcal_m$ is a more suitable object for compression, so we seek a
fast solver to approximate $\Pcal_m(\phi_m)$ for any input $\phi_m$,
by making use of Randomized SVD
(Algorithm~\ref{alg:operator_RSVD}). This method requires us to apply
the operator $\Pcal_m$ to random inputs, which amounts to finding the
local solution in $\Dcal_m$
with randomly constructed boundary conditions,
and then confining it to $\Gamma^\text{s}_{m,+}$. Following
\eqref{eqn:update}, the evaluation $f=\Pcal_m (r)$ is defined as
\[
f = u|_{\Gamma^\text{s}_{m,+}}\quad\text{where $u$ solves}\quad\begin{cases}
v\cdot \nabla_x u(x,v)  = \frac{1}{\vep} \sigma^\delta(x)\Lcal u(x,v) &\quad \text{in} \quad \Dcal_m \\
u(x,v) = r(x,v) &\quad \text{on} \quad \Gamma_{m,-}
\end{cases}\,.
\]
Finding $\Pcal^\ast_m$, the adjoint of $\Pcal_m$, is much more
complicated, as we show in the following theorem, whose proof appears
in the Appendix.

\begin{theorem}\label{thm:adjoint}
The adjoint operator $\Pcal^\ast_m$ is defined by
\begin{equation}
\begin{aligned}\label{eqn:adjoint_boundary_update}
\Pcal^\ast_m : \quad &  L^2(\Gamma^s_{m,+}; |n\cdot v| ) &\rightarrow \quad & L^2(\Gamma_{m,-}; |n\cdot v| ) \\
	& \psi &\mapsto \quad & h|_{\Gamma_{m,-}}
\end{aligned},
\end{equation}
where $h$, supported on $\Dcal_m\backslash \Dcal_m^s$, satisfies:
\begin{equation}\label{eqn:adjoint_RTE1}
\begin{cases}
	(-v\cdot \nabla_x -\frac{\sigma^\delta}{\vep}\Lcal ) h = 0 \quad &\text{in}\quad \Dcal_m\backslash \Dcal_m^s \\
	h = g  \quad &\text{on} \quad \Gamma_{m,-}^s\\
	h = 0  \quad &\text{on} \quad \Gamma_{m,+}\\
\end{cases}\,,
\end{equation}
in which $g$ is the solution to:
\begin{equation}\label{eqn:adjoint_RTE2}
\begin{cases}
	(-v\cdot \nabla_x -\frac{\sigma}{\vep}\Lcal ) g = 0 \quad & \text{in}\quad \Dcal_m^s \\
	g = \psi + h|_{\Gamma^s_{m,+}} \quad & \text{on} \quad \Gamma_{m,+}^s
\end{cases}\,.
\end{equation}
The operators $\Pcal^\ast_m$ and $\Pcal_m$ are adjoint in the sense
that:
\begin{equation}
\langle \Pcal_m \phi, \psi\rangle_{\Gamma_{m,-}^s} = \langle \phi, \Pcal^\ast_m \psi \rangle_{\Gamma^s_{m,+}}\,,
\end{equation}
where $\langle \cdot,\cdot \rangle_{\Gamma_{m,+}^s}$ and $\langle
\cdot,\cdot \rangle_{\Gamma_{m,-}}$ are weighted-$L^2$ inner products
on $L^2(\Gamma^\text{s}_{m,+}; |n\cdot v| )$ and $ L^2(\Gamma_{m,-};
|n\cdot v| ) $, respectively, defined by
\[
	\langle f\,,g \rangle_{\Gamma_{m,+}^s} =
        \int_{\Gamma^\text{s}_{m,+}} fg|n\cdot v|\, \rd{x} \, \rd{v}\,,\quad
        \text{and}\quad \langle f\,,g \rangle_{\Gamma_{m,-}} =
        \int_{\Gamma_{m,-}} fg|n\cdot v| \, \rd{x} \, \rd{v}\,.
\]
\end{theorem}

The computation involved in finding the adjoint operator is
complicated. It requires the computation of two adjoint RTEs over
$\Dcal_m^\text{s}$ and $\Dcal_m\backslash\Dcal_m^\text{s}$,
respectively, that are coupled in a nontrivial fashion through the
boundary conditions, as seen in~\eqref{eqn:adjoint_RTE1}
and~\eqref{eqn:adjoint_RTE2}. We further note that the measure is not
the standard Lebesgue measure, but rather is weighted by $|n\cdot v|$.
%
%

\subsection{Design of the Adjoint Map for $\Scal^\text{s}_m$}\label{sec:adjoint}

Although the operator $\Pcal_m$ is of approximate low-rank, its
adjoint $\Pcal_m^\ast$, which is needed to compute the low-rank
approximation, is complicated. The operator $\Scal_m$, on the other
hand, is not compressible, but its adjoint is relatively easy to
find. In this section, we show that we can approximate $\Pcal_m$ by an
approximately low-rank operator based on $\Scal_m$ whose adjoint is
easy to find.

Since $\Scal_m$ has slow singular decay mainly because it contains too
much information from the boundary layer, we consider a restriction of
this operator from $\Dcal_m$ to $\Dcal_m^{\text{s}}$, which we call
$\Scal^\text{s}_m$. The restriction to $\Dcal_m^{\text{s}}$ eliminates
most of the effects of the boundary layer.
This operator is defined as follows.
\begin{equation}
\begin{aligned}
\Scal^\text{s}_m : \quad &  L^2(\Gamma_{m,-}; |n\cdot v| ) &\rightarrow \quad & L^2(\Dcal^\text{s}_{m}) \\
& \phi &\mapsto \quad & u^\text{s}
\end{aligned}\,,
\end{equation}
where $u_m^\text{s}=u_m|_{\Dcal^\text{s}_{m}}$ and $u_m =
\Scal_m\phi$. The advantages of using this operator are threefold.
\begin{itemize}
\item[1.] $\Pcal_m$ can be defined easily in terms of
  $\Scal^\text{s}_m$.  Nothing is lost by comparison with
  \eqref{eqn:update}; we have
\begin{equation}\label{eqn:update_new}
\mathcal{P}_m:\;\phi_m\xrightarrow{\Scal^\text{s}_m}
u_m^\text{s}\rightarrow u^\text{s}_{m}|_{\Gamma^\text{s}_{m,+}}\,,
\end{equation}
and $u^\text{s}_{m}|_{\Gamma^\text{s}_{m,+}}$ once again serves as the
new boundary condition $\phi_{m\pm1}$, as in
equation~\eqref{eqn:MidSchwarz1}. Note that the trace in
\eqref{eqn:update_new} is well defined, as $\Scal_m$ maps boundary
conditions to $H_A(\Dcal_m)$, so the image of its restriction to
$\Dcal_m^{\text{s}}$ has a trace on the boundary
$\Gamma_{m,+}^{\text{s}}$ of $\Dcal_m^{\text{s}}$.
\item[2.] Because effects from boundary layers are excluded in
  $\Scal_m^\text{s}$, it can be expected to have approximate low rank.
  Figure~\ref{fig:svd_patch3} shows that the decay rate of
  $\Scal^\text{s}_m$ (upon discretization) is almost the same as for
  $\Pcal_m$.
\begin{figure}[t]
\centering
\includegraphics[width = 0.6\textwidth,height = 0.25\textheight]{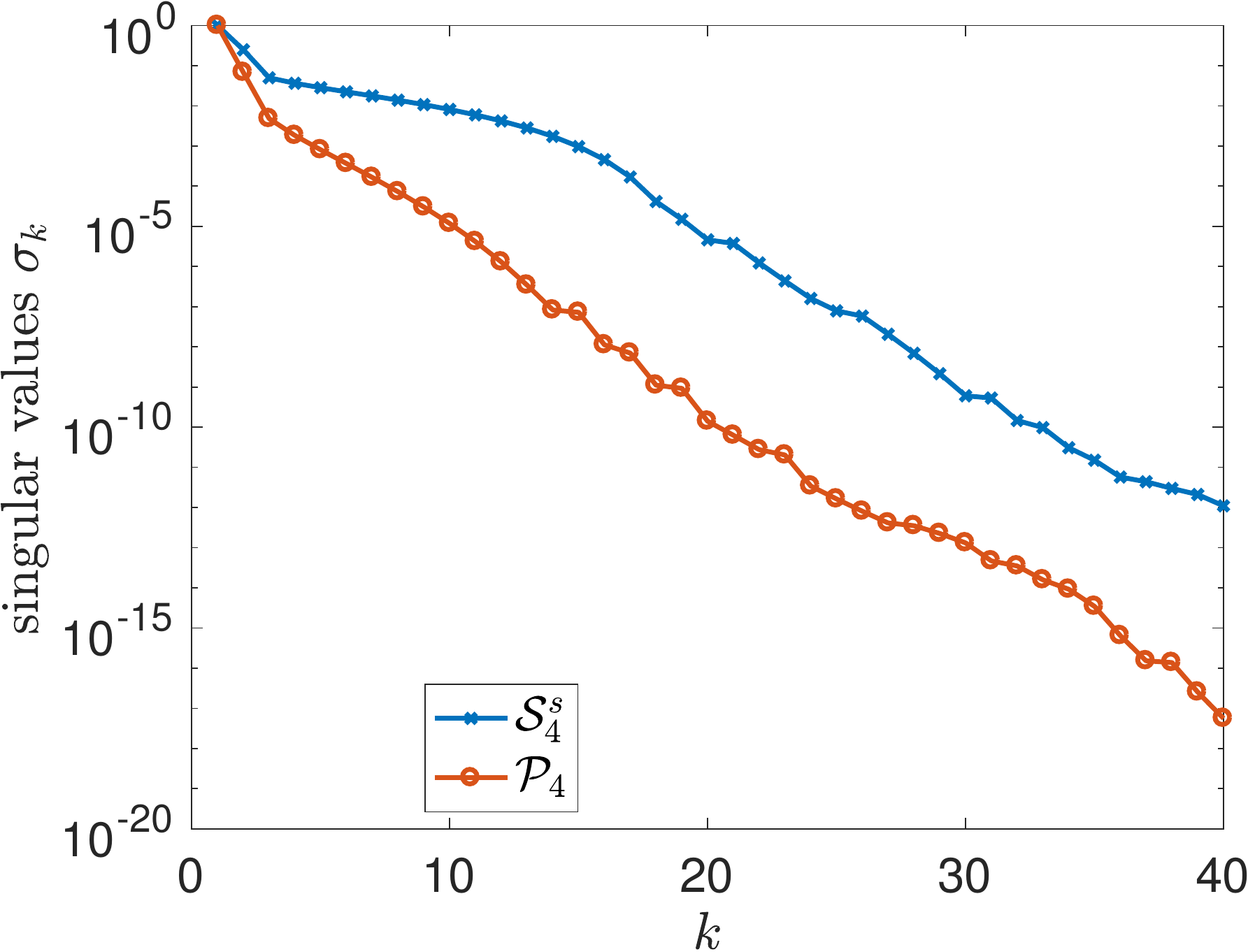}
\caption{Singular values of $\Scal_4^s$ and $\Pcal_4$ when $(\vep,\delta) = (1/81,1/81)$.}
\label{fig:svd_patch3}
\end{figure}
\item[3.] The adjoint is easy to compute, as we show next, in
  Theorem~\ref{thm:adjoint_S}.
\end{itemize}

\medskip

\begin{theorem}\label{thm:adjoint_S}
  The adjoint of $\Scal_m^\text{s}$ is defined as follows:
\begin{equation}\label{eqn:map_ad}
\begin{aligned}
\left(\Scal_m^{\text{s}}\right)^\ast: \ &L^2(\Dcal^\text{s}_{m}) \ &\rightarrow \ &L^2(\Gamma_{m,-};|n\cdot v|)\\
&g \ &\mapsto \ & h|_{\Gamma_{m,-}}
\end{aligned},
\end{equation}
where $h$ solves the adjoint RTE over $\Dcal_m$, which is 
\begin{equation}\label{eqn:RTE_ad}
\begin{cases}
(-v\cdot \nabla_x-\frac{1}{\vep}\sigma^\delta(x)\Lcal) h = \tilde{g} \quad & \text{in} \quad \Dcal_m \\
h = 0\ \quad & \text{on}\quad \Gamma_{m,+}
\end{cases}\,,
\end{equation}
and the source $\tilde{g}$ is the trivial extension of $g$ over
$\Dcal_m$, that is,
\[
\tilde{g} = g\, \;\; \mbox{for $(x,v)\in\Dcal^\text{s}_m$} \quad
\text{and}\quad\tilde{g} = 0  \;\; \mbox{for  $(x,v)\in\Dcal_m\backslash
\Dcal_m^\text{s}$}.
\]
\end{theorem}
\begin{proof}
We need to show that
\[
	\langle g, \Scal^\text{s}_{m} \phi\rangle_{\Dcal_m^\text{s}} = \langle\left(\Scal_m^{\text{s}}\right)^\ast g \,,\phi\rangle_{\Gamma_{m,-}}
\]
for all $g$ and $\phi$. Denoting by $u$ the solution
to~\eqref{eqn:RTE_local} with boundary condition $\phi$, then the
definition of $\Scal^\text{s}_m$ implies that the left hand side of
this expression is
\[
	\langle g, \Scal^\text{s}_{m} \phi\rangle_{\Dcal_m^\text{s}}  = \int_{\Dcal^\text{s}_{m}}g u\ \, \rd x \, \rd v = \int_{\Dcal_m} \tilde{g} u \, \rd x \, \rd v\,.
\]
Denoting by $h$ the solution to~\eqref{eqn:RTE_ad} with source term
$g$, we have
\begin{equation*}
\begin{aligned}
	\int_{\Dcal_m} \tilde{g} u\ \, \rd x \, \rd v &= \int_{\Dcal_m} u(-v\cdot \nabla_x-\frac{1}{\vep}\sigma^\delta\Lcal) h \ \, \rd x \, \rd v \\
	& = -\int_{\Gamma_{m,-}} uh v\cdot n \, \rd{x}\, \rd{v}- \int_{\Gamma_{m,+}} uh v\cdot n \, \rd{x}\, \rd{v}  \\
    & \qquad \qquad + \int_{\Dcal_m} h(v\cdot \nabla_x-\frac{1}{\vep}\sigma^\delta\Lcal)u \, \rd{x}\, \rd{v}\\
	& = \int_{\Gamma_{m,-}} uh |v\cdot n| \, \rd{x}\, \rd{v} \\
	& = \langle \left(\Scal_m^{\text{s}}\right)^\ast g,\phi \rangle_{\Gamma_{m,-}}\,,
\end{aligned}
\end{equation*}
yielding the desired result.
\end{proof}

By comparing Theorem~\ref{thm:adjoint} with
Theorem~\ref{thm:adjoint_S}, we see immediately that computing the
adjoint operator $\left(\Scal_m^{\text{s}}\right)^\ast$ is
significantly easier than computing $\Pcal^\ast_m$.

\subsection{Low-Rank Schwarz Iteration Method}\label{sec:reduced_schwarz}

We can use $\left(\Scal_m^{\text{s}}\right)^\ast$ to implement the
RSVD method to find the low-rank approximation to the operator
$\Scal_m^\text{s}$. Given target rank $r$, and denoting the reduced
operator by $\Scal_{m,r}^\text{s}$, we look for functions $\mu_i$ and
$\nu_i$, and nonnegative scalars $\sigma_i$, such that:
\begin{equation}\label{eqn:S_reduce}
\Scal_{m,r}^\text{s} =
\sum_i\sigma_i\mu_i(x_1,v_1)\nu_i(x_2,v_2)\,,\quad \mbox{for all
  $(x_1,v_1)\in\Dcal_m^\text{s}$ and $(x_2,v_2)\in\Gamma_{m,-}$,}
\end{equation}
where $\mu_i(x_1,v_1)$ and $\nu_i(x_2,v_2)$ are obtained in
Algorithm~\ref{alg:rte_rsvd}.
\begin{algorithm}[t]
		\caption{Approximation of $\Scal_{m}^\text{s}$ via RSVD}\label{alg:rte_rsvd}	
		\begin{algorithmic}[1]
			\State Given desired rank $r$ and oversampling parameter $p$, set $k:=r+p$;
			\State \textbf{Stage I}
			\Indent
			\State Generate $k$ independent Gaussian test vectors $\omega_1,\ldots,\omega_k$;
			\State Prepare incoming boundary conditions $\wt{w_i} = \Emat w_i\,, i=1,\ldots,k$, where $\Emat = [e_1,e_2,\dotsc]$ collects discrete orthonormal basis functions in $L^2(\Gamma_{m,-};|n\cdot v|)$;
			\State Evaluate $u_{m,i}|_{\Dcal_m^s} = \Scal_{m}^\text{s} \wt{w}_i\,, i=1,\ldots,k$ by solving \eqref{eqn:RTE_local} with boundary conditions $\wt{w}_i$, then taking restrictions over $\Dcal_{m}^s$;
			\State Construct matrix $\Qmat=[q_1,\ldots,q_k]$ whose columns form an orthogonal basis for span$\{u_{m,1}|_{\Dcal_m^s},\ldots,u_{m,k}|_{\Dcal_m^s}\}$;
			\EndIndent
			\State \textbf{Stage II}
			\Indent
			\State Prepare sources $g_i = [q_i \quad 0]$, $i=1,2,\dotsc,k$, so that $g_i = q_i$ over $\Dcal_{m}^s$ and $g_i=0$ over $\Dcal_m\backslash \overline{\Dcal_{m}^s}$;
			\State Evaluate $b_i = Yg_i$ by solving \eqref{eqn:RTE_ad} with $g_i$ as source, for $i=1,2,\dotsc,k$, then taking restrictions over $\Gamma_{m,-}$;
			\State Form matrix $\Bmat = [b_1,\ldots,b_k]$;
			\State Compute SVD of $\Bmat = \wt{\mathsf{M}}_k\Sigma_k\mathsf{N}_k^\ast$, where
			\begin{equation*}
			\wt{\mathsf{M}}_k = 
			\begin{bmatrix}
			\wt{\mu}_1,\ldots,\wt{\mu}_k
			\end{bmatrix}\,,\quad 
			\Sigma_k = \text{diag}\{\sigma_1,\ldots,\sigma_k\}\,,\quad \text{and} \quad 
			\mathsf{N}_k = 
			\begin{bmatrix}
			\nu_1,\ldots,\nu_k
			\end{bmatrix};
			\end{equation*}
			\State Compute $\mathsf{M}_k := \Qmat \wt{\mathsf{M}}_k$ and denote $\mathsf{M}_k =\begin{bmatrix}
			\mu_1,\ldots,\mu_k
			\end{bmatrix}$;
			\EndIndent
			\State \textbf{Return:} $\Scal_{m,r}^\text{s} = \sum_{i=1}^r \sigma_i\mu_i(x_1,v_1)\nu_i(x_2,v_2)$ for $(x_1,v_1)\in\Dcal_m^\text{s}$ and $(x_2,v_2)\in\Gamma_{m,-}$.
		\end{algorithmic}
\end{algorithm}

We note that in this algorithm, $k$ and $r$ could be hard to choose ahead of time. Numerically we can choose it ``on-the-fly''. This means we simply set an accuracy threshold and stop the process once the newly computed $u_{m,i}|_{\Dcal_m^s}$ falls almost in the previous generated space within the preset error tolerance. In terms of the low-rank operator $\Scal^\text{s}_{m,r}$, the procedure~\eqref{eqn:update_new} is reduced further to
\begin{equation}\label{eqn:update_reduce}
\mathcal{P}_{m,r}:\;\phi_m\xrightarrow{\Scal^\text{s}_{m,r}}
u^\text{s}_{m}\rightarrow u^\text{s}_{m}|_{\Gamma^\text{s}_{m,+}}\,,
\end{equation}
and we once again use $u^\text{s}_{m}|_{\Gamma^\text{s}_{m,+}}$ to
obtain the solutions $\phi_{m\pm1}$ to be used at the next time step,
as in equation~\eqref{eqn:MidSchwarz1}.


The procedure we have just outlined provides a much cheaper way to
evaluate $\Pcal$ in~\eqref{eqn:update} for the following reasons.
\begin{enumerate}
\item $\Scal_m^\text{s}$ maps the boundary condition to the interior
  of the subdomain, and it is cheaper to evaluate than $\Scal_m$,
  whose range has a bigger support.
\item The format in \eqref{eqn:S_reduce} guides the evaluation of
  $\Scal_{m,r}^\text{s}(\phi_m)$; we have
\begin{equation}\label{eqn:Scal_mr_evaluate}
\Scal_{m,r}^\text{s}(\phi_m) =
\sum_{i=1}^k\sigma_iu_i(x_1,v_1)\int_{\Gamma_{m,-}}\phi_m(x_2,v_2)v(x_2,v_2)|n_{x_2}\cdot
v_2| \, \, \rd{x_2}\, \rd{v_2}\,.
\end{equation}
This evaluation requrires $\mathcal{O}(k|\Gamma_{m,-}|)$ operations,
where $|\Gamma_{m,-}|$ is the cardinality (the number of grid points)
in $\Gamma_{m,-}$.
\end{enumerate}

Algorithm~\ref{alg:re_Schwarz} summarizes the complete approach using
the reduced solution map. The method is divided into offline and
online stages. The reduced operators $\Scal^\text{s}_{m,r}$ are found
in the offline stage, then called repeatedly in the online stage,
during the Schwarz iteration procedure.
\begin{algorithm}
	\caption{Low-Rank Schwarz Method}\label{alg:re_Schwarz}
	\begin{algorithmic}[1]
		\State \textbf{Offline Stage:}
		\Indent
			\State Call Algorithm \ref{alg:rte_rsvd} for all local reduced solution maps $\Scal_{m,r}^\text{s}$, with $m=1\,,\dotsc,M$;
		\EndIndent
		\State \textbf{Online Stage:}
		\Indent
			\State \textbf{Input:} global boundary conditions $\phi$ in \eqref{eqn:RTE_bd} and error tolerance $\tau$;
			\State Set $t \leftarrow 0$ and initiate all inflow boundary conditions from  \eqref{eqn:InitSchwarz0} and \eqref{eqn:InitSchwarz1};
			\State \textbf{While $\text{error}>\tau$} 
			\Indent
			\State $t=t+1$;
			\State	\textbf{For $m = 1,\ldots,M$}
			\Indent
			\State $u^\text{s}_m \leftarrow \Scal_{m,r}^\text{s}(\phi_m^{t-1})$ according to~\eqref{eqn:Scal_mr_evaluate};
			\State $\phi_{m\pm 1}^t = u^\text{s}_m|_{\Ecal_{m,m\pm 1}}$;
			\EndIndent
			\State \textbf{EndFor}
			\State $\text{error} = \sum_{m}\|\phi_m^t-\phi_m^{t-1}\|$;
			\EndIndent
			\State \textbf{EndFor}
			\State	\textbf{For $m = 1,\ldots,M$}
			\Indent
			\State $u_m \leftarrow \Scal_{m}(\phi_m^{t})$;
			\EndIndent
			\State\textbf{EndFor}
			\State Assemble the final solution using~\eqref{eqn:assemble};
			\State \textbf{Return:} final solution $u^\text{final}$.
		\EndIndent
	\end{algorithmic}
\end{algorithm}

%

\section{Numerical examples}\label{sec:Numerical}
In this section, we present numerical examples to validate the
accuracy and efficiency of our methods. We consider boundary value
problem \eqref{eqn:RTE_bd} with domain $\Dcal = \Kcal\times \Vcal =
(0,1)\times (-1,1)$ and highly oscillatory scattering coefficient
$\sigma^\delta(x)$ defined by
\begin{equation} \label{eq:sd4}
\sigma^\delta(x) = \frac{1.1 + \cos(4\pi x)}{1.1 + \sin(2\pi
  x/\delta)} \in [0.047,21],
\end{equation}
where $\delta$ represents the period of oscillation in the spatial
space. See Figure~\ref{fig:media_d81} for a graph of
$\sigma^\delta(x)$ with $\delta = 1/81$.
\begin{figure}[h]
	\centering
	\includegraphics[width=0.6\textwidth]{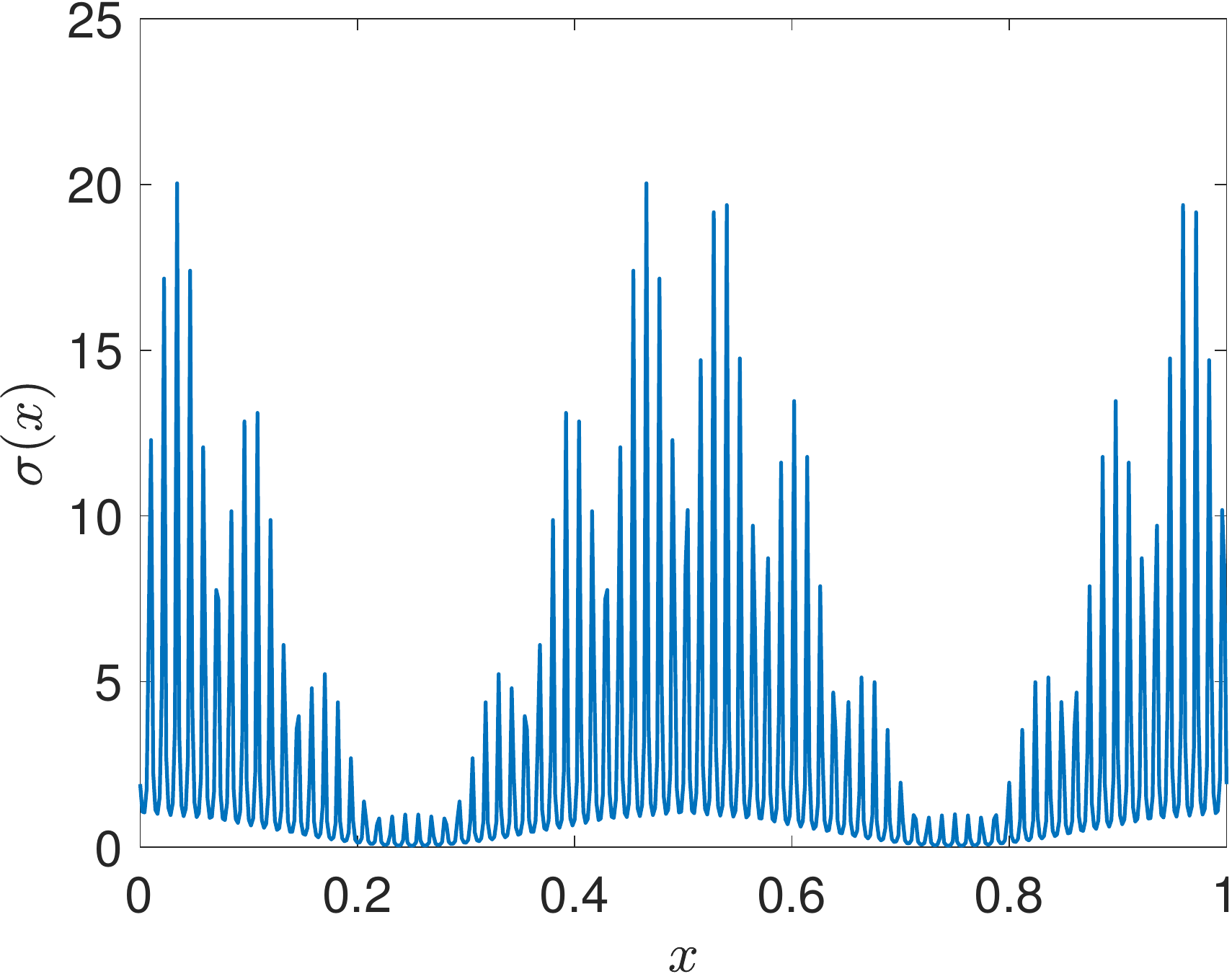}
	\caption{Graph of oscillatory media with $\delta = 1/81$.}
	\label{fig:media_d81}
\end{figure}

The space domain $\Kcal$ is divided into $M = 10$ different local
subdomains $\Kcal_m, m=1,\ldots,M$, as follows:
\[
\Kcal_1 = \left(0,\tfrac{3}{2M} \right)\,, \quad \Kcal_m =
\left(\tfrac{2m-3}{2M},\tfrac{2m+1}{2M} \right)\,, m=2,\ldots,M-1\,,
\quad \Kcal_{M} = \left(1-\tfrac{3}{2M},1 \right),
\]
so that each subdomain $\Kcal_m$ overlaps with its neighboring
subdomains $\Kcal_{m-1}$ and $\Kcal_{m+1}$ (except for the subdomains
$\Kcal_1$ and $\Kcal_M$ at the two ends of the domain, which overlap
with just one neighbor each). The subdomains $\Kcal_{m}^\text{s}$
defined in \eqref{eq:Dsm} are 
\[
\Kcal_{1}^\text{s} = \left(\tfrac{1}{2M},\tfrac{3}{2M} \right)\,, \quad
\Kcal_{m}^\text{s} = \left(\tfrac{m-1}{M},\tfrac{m}{M} \right)\,,
m=2,\ldots,M-1\,, \quad \Kcal_{M}^\text{s} =
\left(1-\tfrac{3}{2M},1-\tfrac{1}{2M} \right).
\]
We thus have $\Ecal_{m, m\pm 1} = \{\frac{m}{M}\}\times \Vcal\,, \quad
m=1,\ldots,M-1$, while $\Dcal_{m}^\text{s}: = \Kcal_{m}^\text{s}\times
\Vcal$ satisfies 
\[
\Ecal_{m, m\pm 1} \subset \Dcal_{m}^\text{s} \subset\Dcal_m\,, m = 1,\dotsc,M-1 \,.
\]
The spatial domain $\Kcal$ is discretized with a extreme fine mesh of
size $\Delta x = 1/360$, while the velocity domain $\Vcal$ is
discretized with a mesh of size $\Delta v = 2/40 = 0.05$. The fine-mesh discretization is determined by $\delta$ and $\vep$.

\subsection{Local Tests}

We first show that the singular values of the local solution map
$\Scal_{m}^\text{s}$ indeed decay rapidly for small Knudsen number
$\vep$ and small $\delta$. Figure~\ref{fig:local_map_svd} plots the
singular values of $\Scal_{4}^\text{s}$ (relative to the largest
singular value) for various values of $(\vep,\delta)$.  In the case of
large values $\vep=\delta = 1$, the singular values decay slowly and
low-rank structure is not present.  By contrast, in the small-value
regimes $(\vep,\delta)=(1/81,1/81)$ and $(\vep,\delta)=(1/81,1/9)$,
low-rank structure is evident. As a consequence, only half or even a
quarter of basis functions are needed to achieve high accuracy in
approximating the local solution map $\Scal_{m}^\text{s}$.


\begin{figure}
	\centering
	\includegraphics[width=0.6\textwidth]{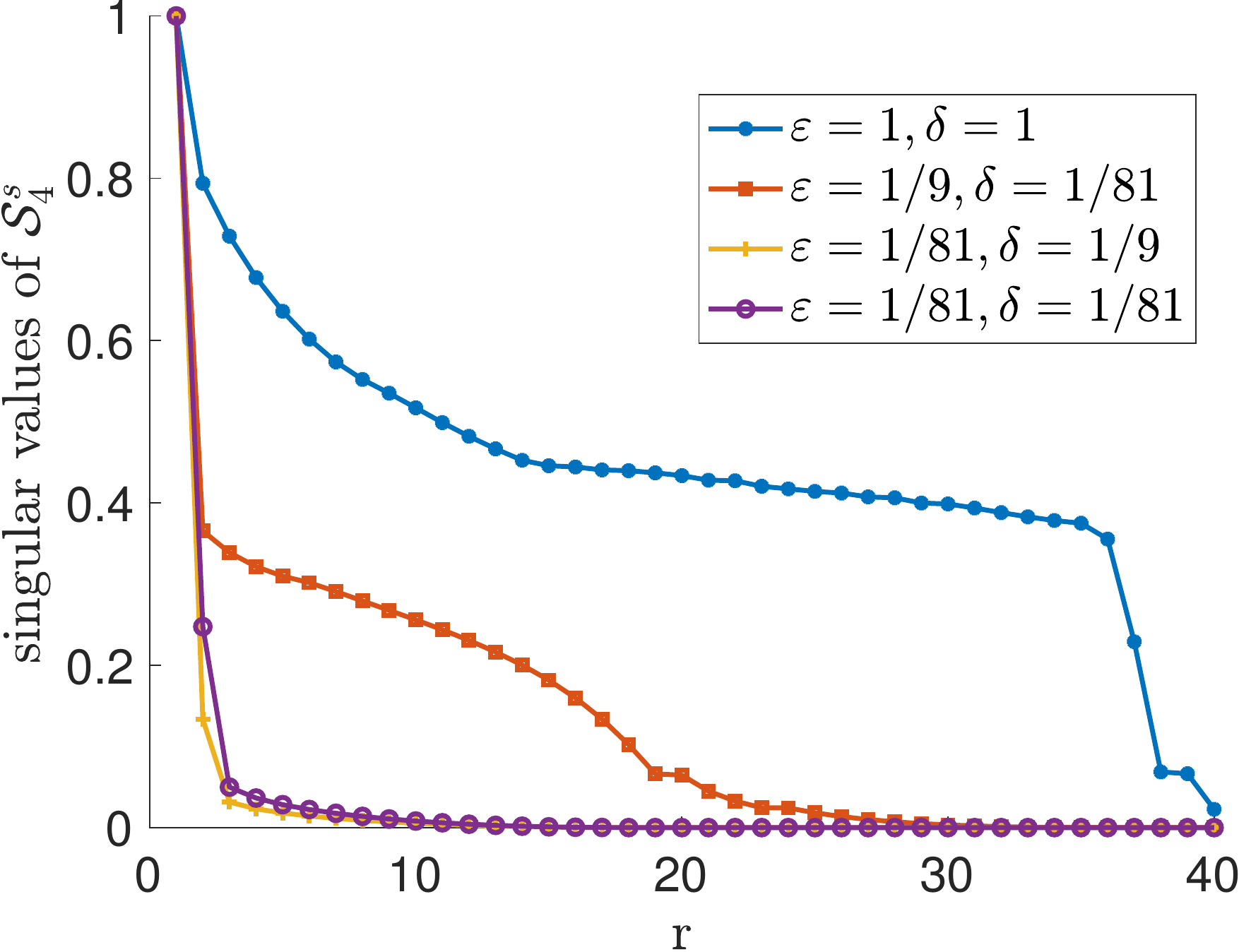}
	\caption{Singular value  of $\Scal_{4}^\text{s}$,
		relative to the largest singular value, plotted for various
		values of the parameter pair $(\vep,\delta)$. The singular
		values decay slowly for $(\vep,\delta) = (1,1)$ and relatively faster for $(\vep,\delta)= (1/9,1/81)$. A much
		faster decay is observed in other limiting regimes of
		$(\vep,\delta)$ approaching $(0,0)$.}
	\label{fig:local_map_svd}
\end{figure}

\subsection{Global Tests}

We consider solving RTE \eqref{eqn:RTE_bd} with scattering parameter
\eqref{eq:sd4} and the following inflow boundary conditions over
$\Gamma_-$:
\begin{equation}\label{eqn:bdy_cond}
\phi(x,v) = \begin{cases}
 10 + \sin(2\pi v)\,, \;\;&\text{at}\; x=0\,, v>0 \\
1 + \sin(2\pi v)\,, \;\; &\text{at}\; x=1\,, v<0.
\end{cases}
\end{equation}
We approximate $\Scal_{m}^\text{s}$ by low-rank operators
$\Scal_{m,r}^\text{s}$, according to Algorithm~\ref{alg:rte_rsvd},
with $r = 2,3,4,5,6$.
We then use these low-rank approximations in the reduced Schwarz method,
Algorithm~\ref{alg:re_Schwarz}.
The various approximating solutions are then compared to the reference
solution, also obtained by the Schwarz method, in terms of accuracy
and speed, for different values of the parameter pair $(\vep,\delta)$. We also document the global error as a function of the number of iterations.

\subsubsection*{Accuracy of approximating solution}

Figures~\ref{fig:global_plot}, \ref{fig:global_plot_3}, and
\ref{fig:global_plot_4} show the reference solution over domain
$\Dcal$ and compare with approximate solutions for $r=2$ and $r=6$,
for parameter pair settings $(\vep,\delta) = (1/81,1/81)$,
$(1/81,1/9)$, and $(1,1)$, respectively.  In
Figure~\ref{fig:global_plot}, for $(\vep,\delta) = (1/81,1/81)$, the
approximate solutions are very close to the reference
solution. Figure~\ref{fig:global_plot_3}, with
$(\vep,\delta)=(1/81,1/9)$, shows poor approximation for $r=2$ but
good approximation for $r=6$. For the large-value case $(\vep,\delta)
= (1,1)$, shown in Figure~\ref{fig:global_plot_4}, both approximations
are poor, due to the lack of low-rank structure in
$\Scal_{m}^\text{s}$.


Figure~\ref{fig:global_err} shows the relative difference between
approximate and reference solutions, plotted as a function of $r$, for
the three settings of $(\vep,\delta)$ considered here.  The difference
is defined by the formula
\begin{equation*}
	\text{Relative Error} = \frac{\|u_{\text{approx}} - u_{\text{ref}}\|_2 }{\|u_{\text{ref}}\|_2} \,,
\end{equation*}
where $u_{\text{approx}}$ is the approximate solution in question and
$u_{\text{ref}}$ is the numerical reference solution computed with fine mesh. We evaluate the difference using $l_2$ norm of the two vectors. We see that the quality
of the approximate solution aligns with the local singular value decay
shown in Figure~\ref{fig:local_map_svd}. For large value case
$(\vep,\delta)=(1,1)$, there is no decay in relative errors as $r$
increases. 
For $(\vep,\delta)=(1/81,1/81)$, the relative error is below $10\%$ for
$r=3$ and decreases as $r$ increases. For $(\vep,\delta)=(1/81,1/9)$,
the relative error decreases rapidly with $r$.

\begin{figure}
	\includegraphics[width=0.3\textwidth]{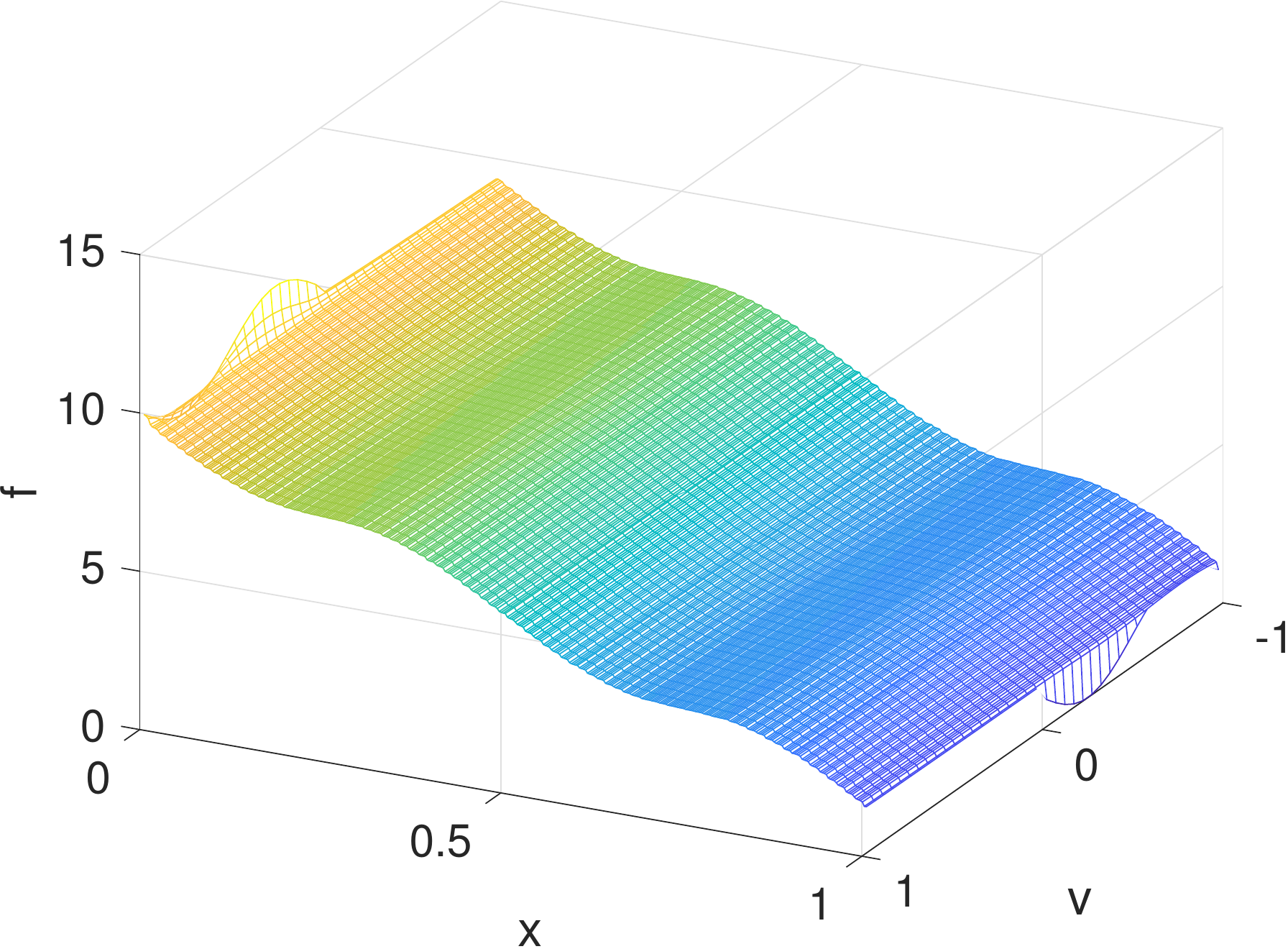}
	\includegraphics[width=0.3\textwidth]{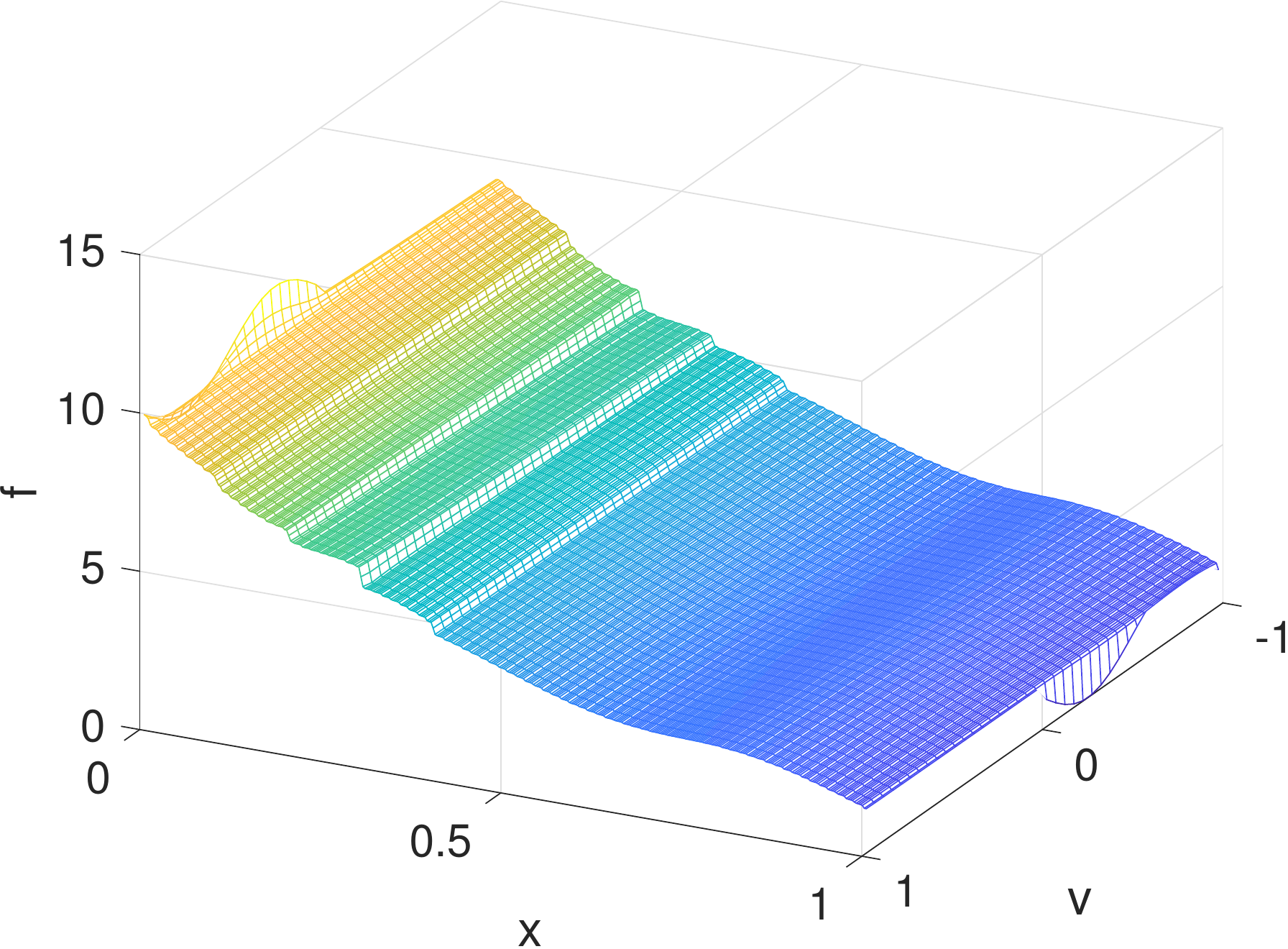}
	\includegraphics[width=0.3\textwidth]{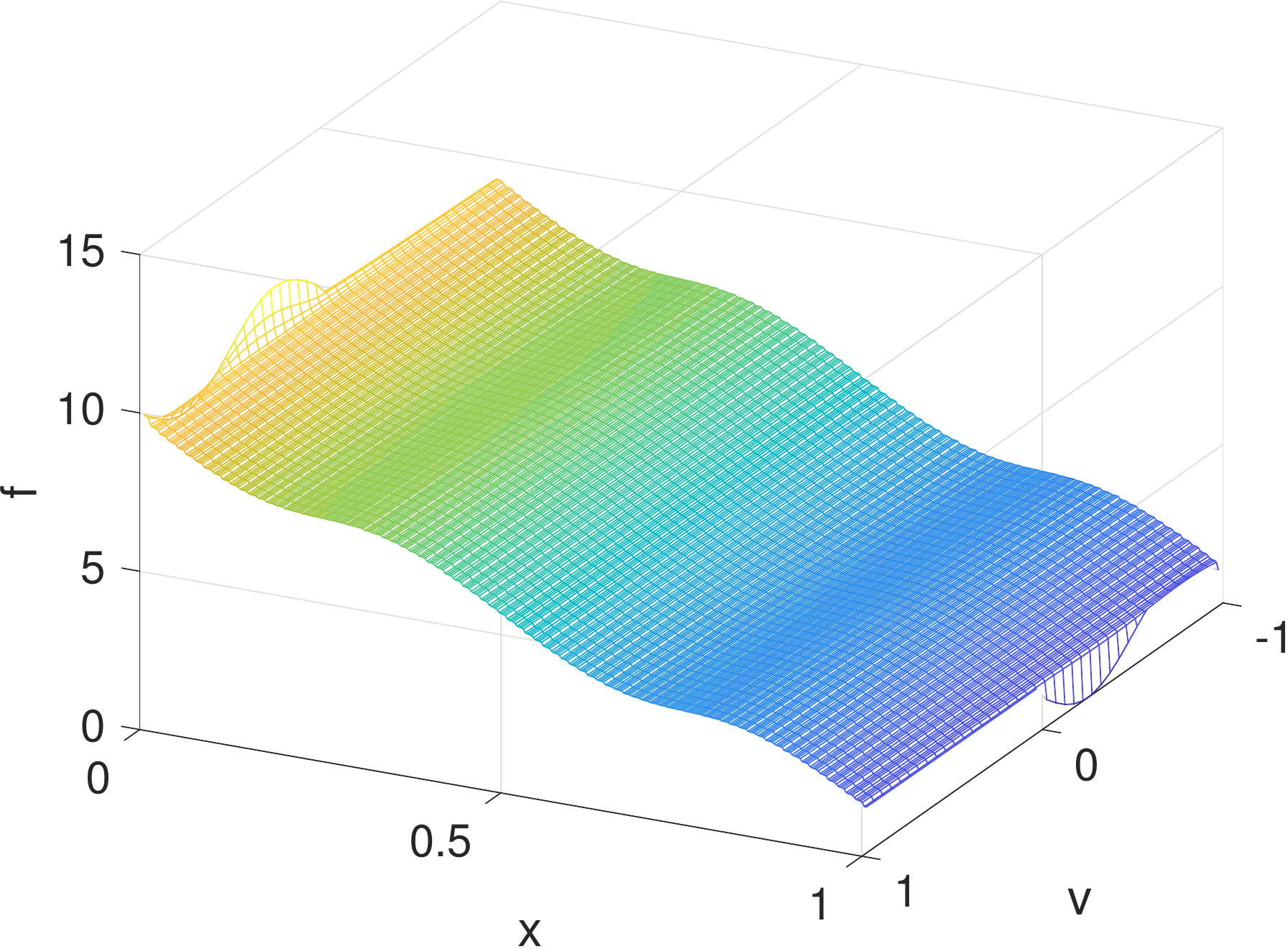}
	\caption{\textbf{For $\vep = \delta = 1/81$. Left:} reference solution, \textbf{Middle:} approximating solution with $r=2$, \textbf{Right:} approximating solution with $r=6$.}
	\label{fig:global_plot}
\end{figure}
\begin{figure}
	\includegraphics[width=0.3\textwidth]{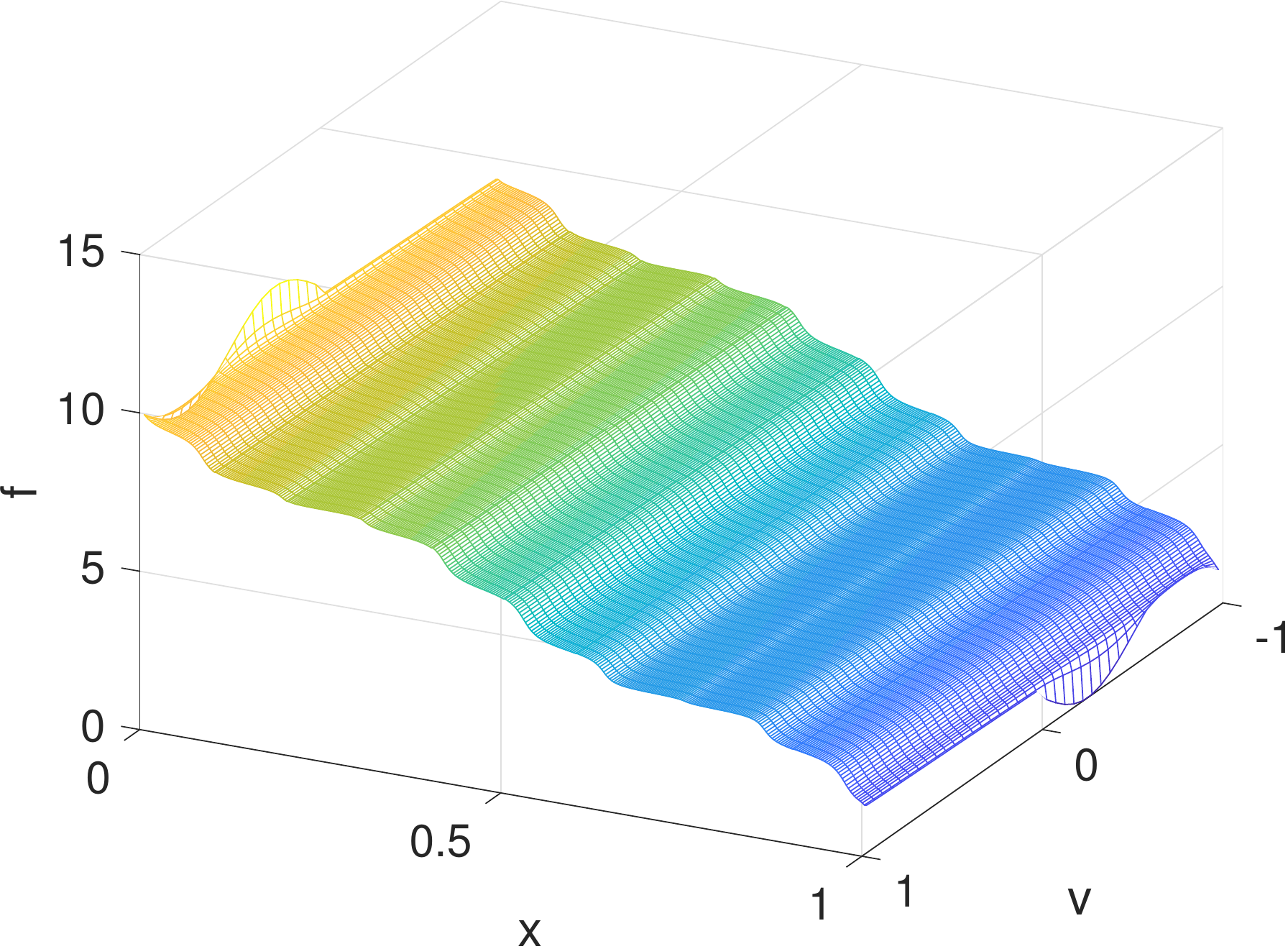}
	\includegraphics[width=0.3\textwidth]{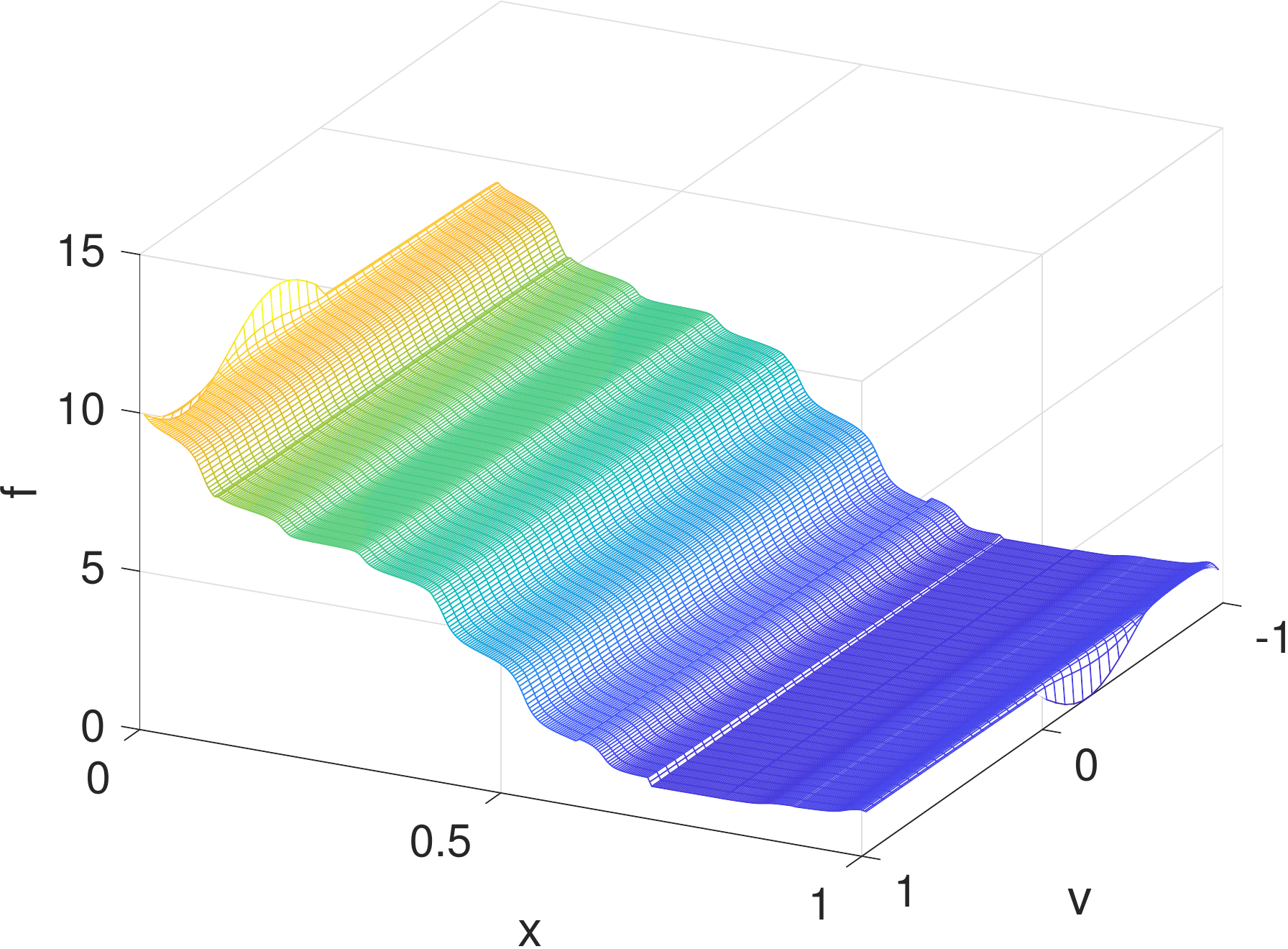}
	\includegraphics[width=0.3\textwidth]{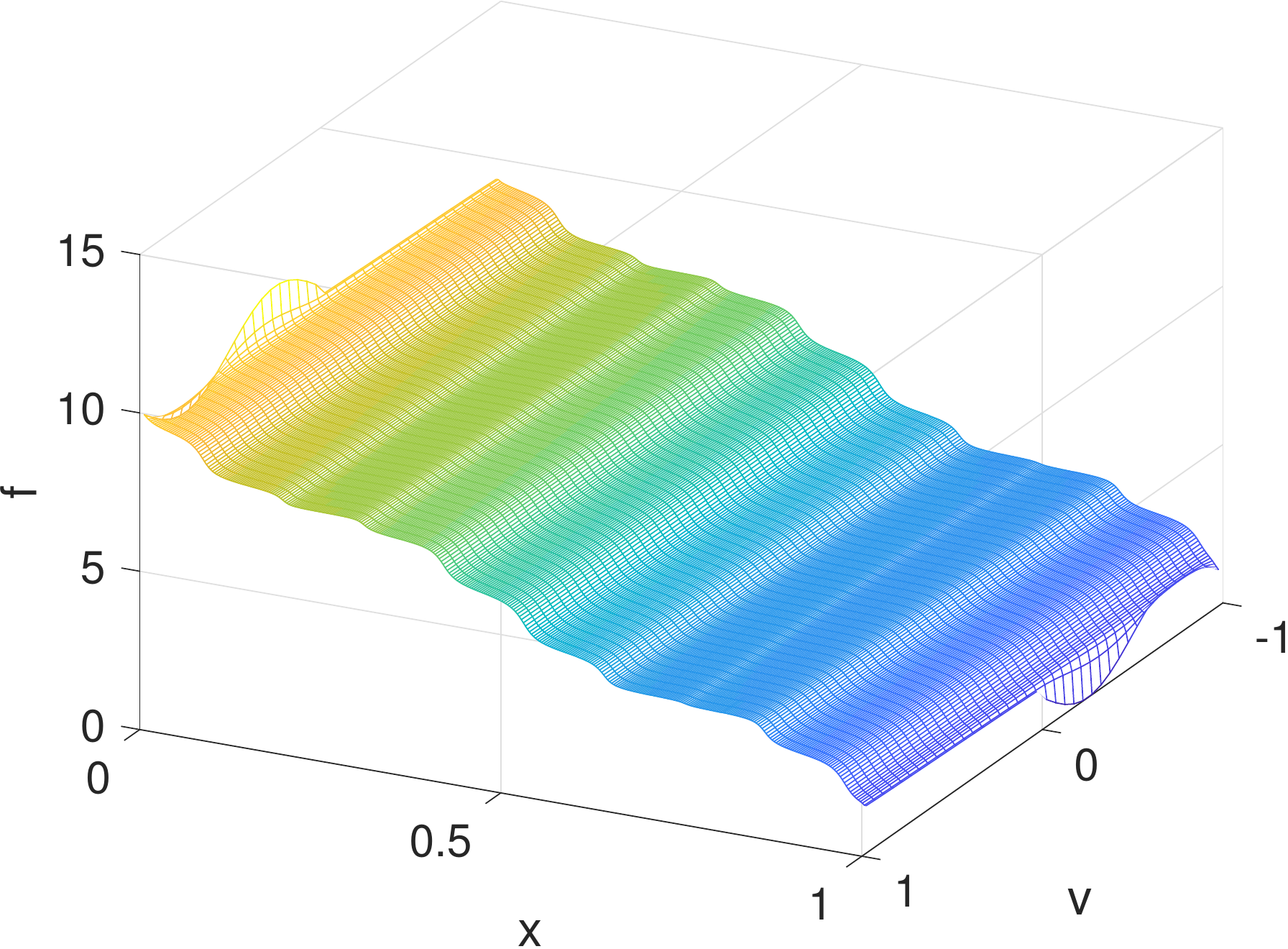}
	\caption{\textbf{For $\vep = 1/81,\delta = 1/9$. Left:} reference solution, \textbf{Middle:} approximating solution with $r=2$, \textbf{Right:} approximating solution with $r=6$.}
	\label{fig:global_plot_3}
\end{figure}
\begin{figure}
	\includegraphics[width=0.3\textwidth]{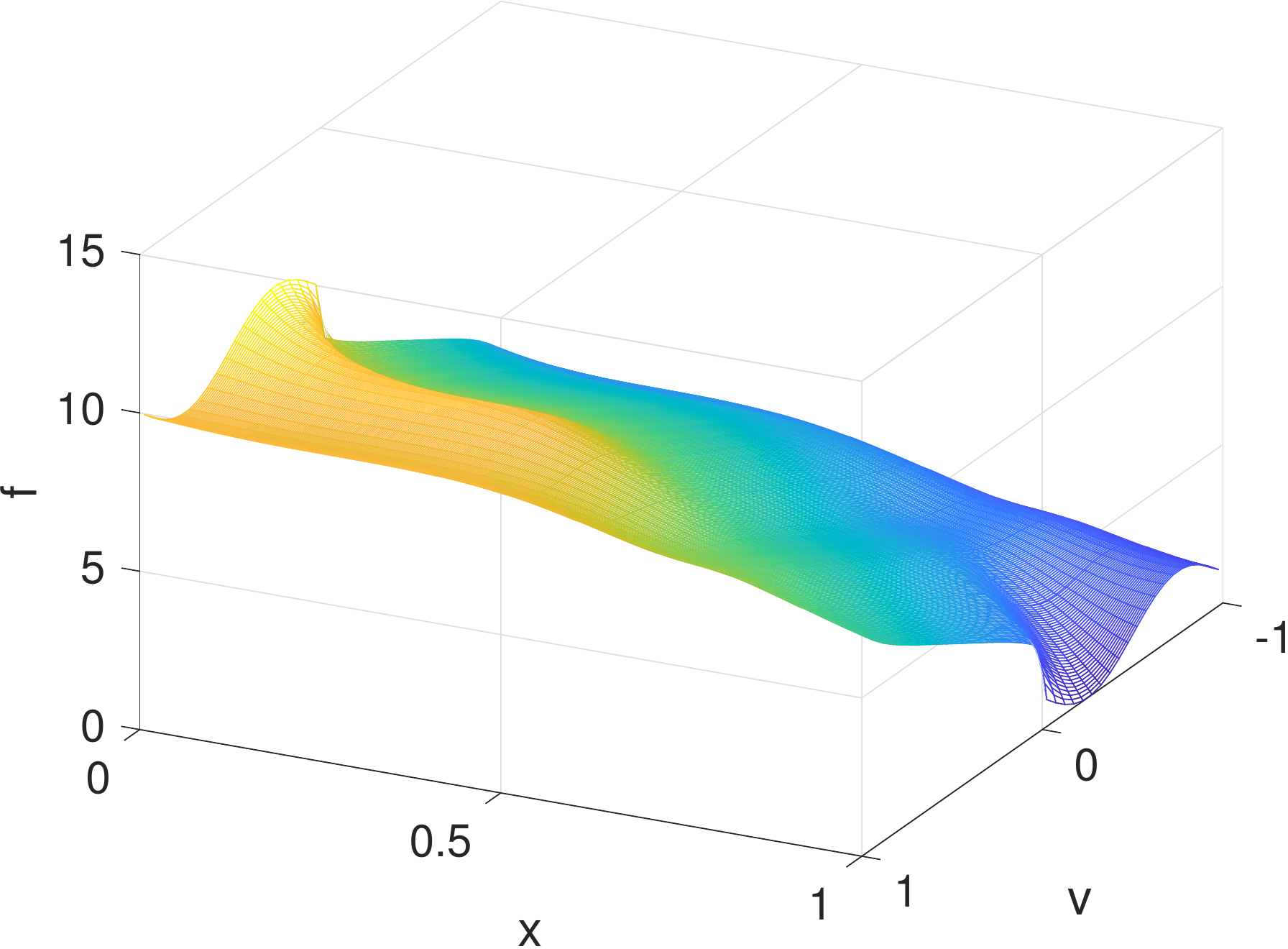}
	\includegraphics[width=0.3\textwidth]{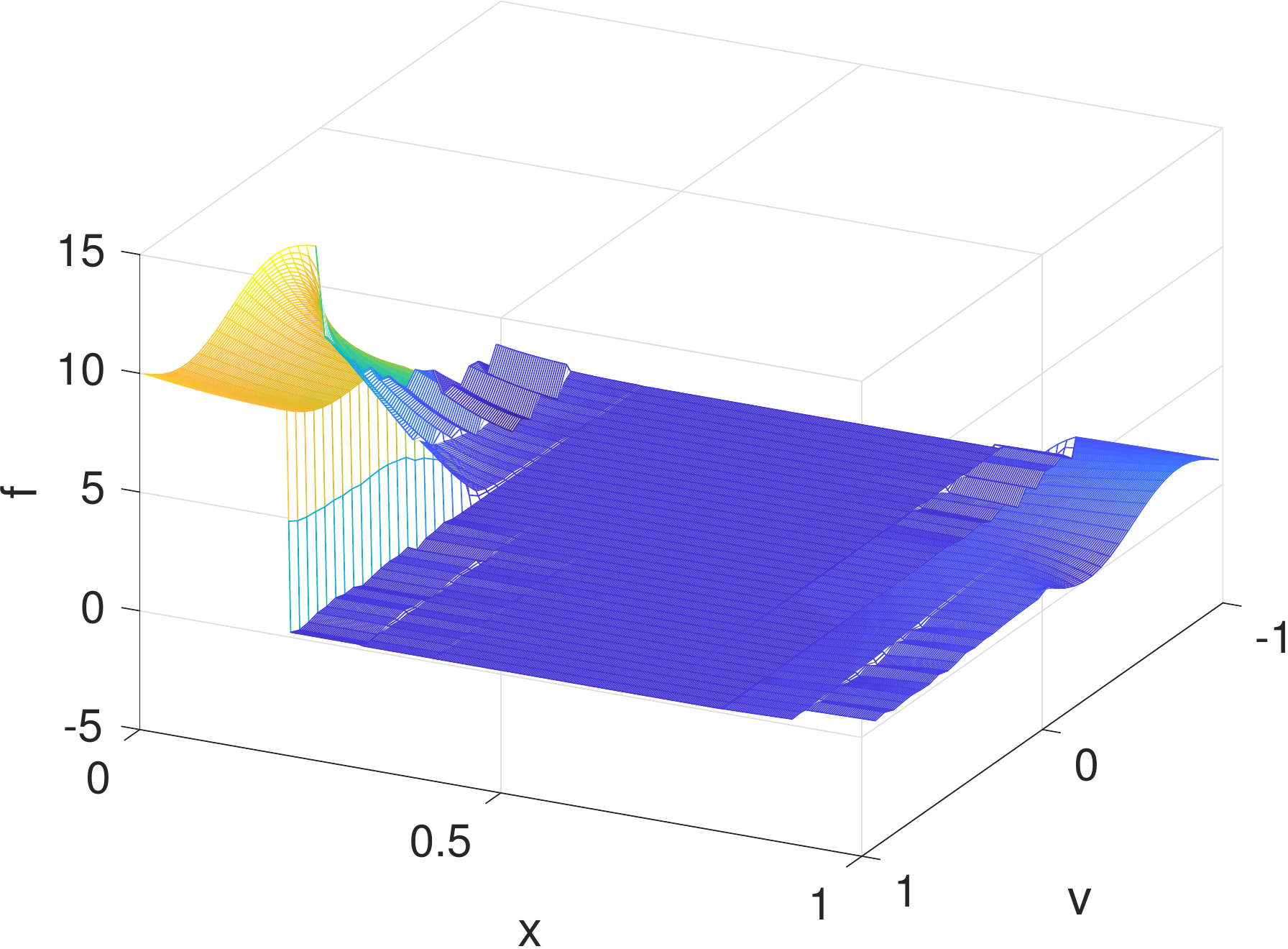}
	\includegraphics[width=0.3\textwidth]{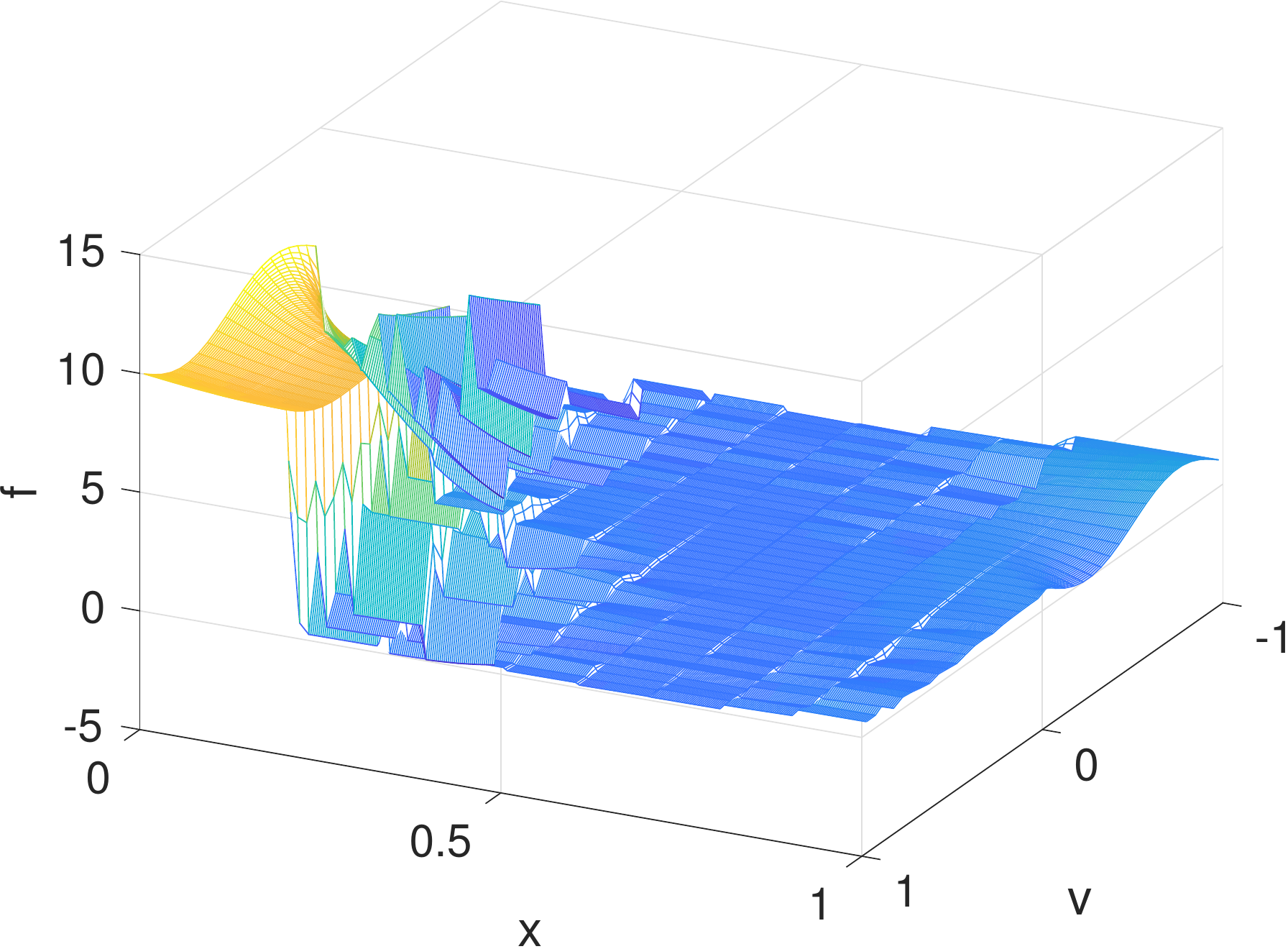}
	\caption{\textbf{For $\vep = \delta = 1$. Left:} reference solution, \textbf{Middle:} approximating solution with $r=2$, \textbf{Right:} approximating solution with $r=6$.}
	\label{fig:global_plot_4}
\end{figure}
\begin{figure}
	\centering
	\includegraphics[width=0.6\textwidth]{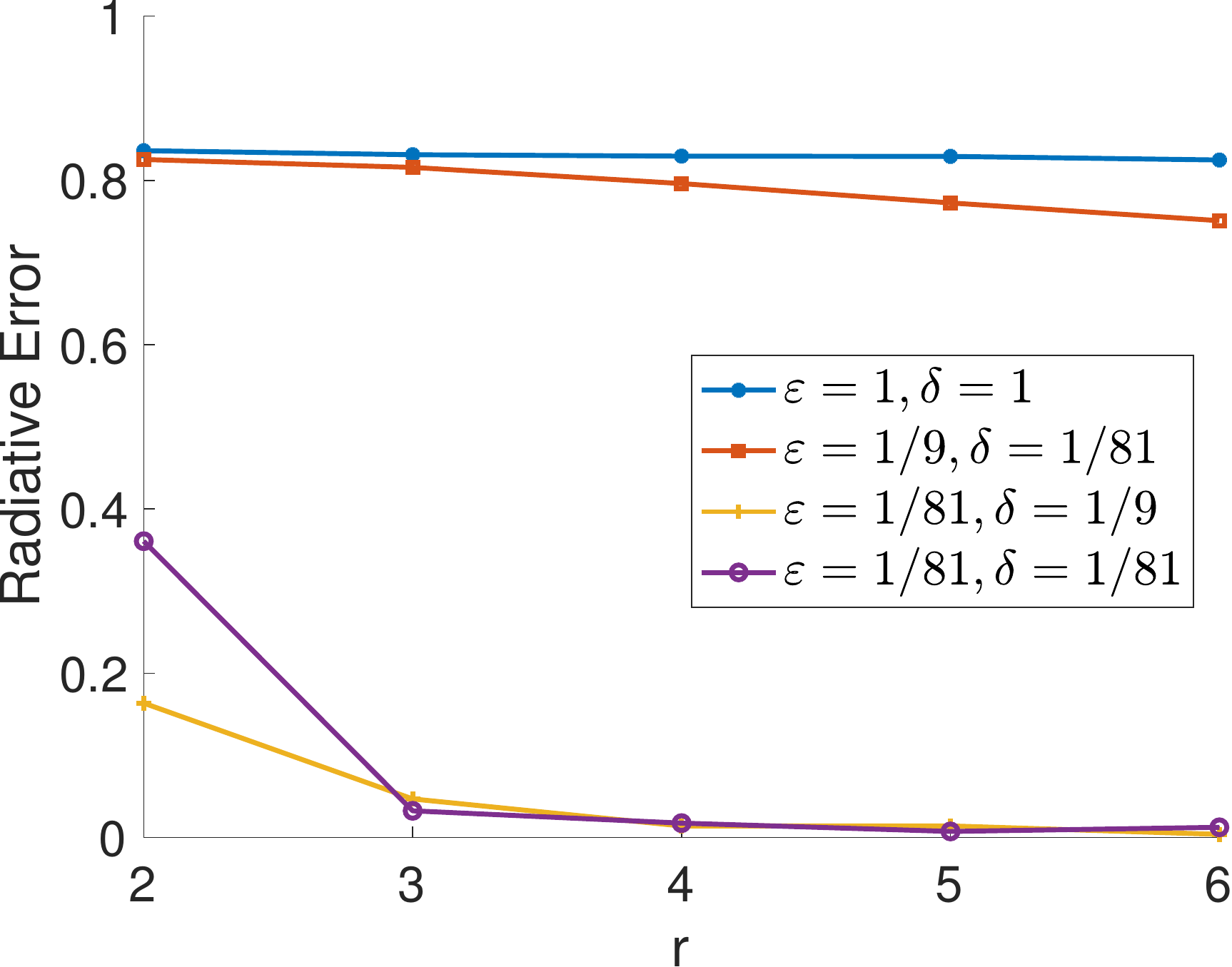}
	\caption{Relative difference between reference solution and
		approximate solutions for various values of $(\vep,\delta)$. 
		The relative error for $(\vep,\delta) = (1/81,1/9)$ with $r= 2,3,4,5,6$ is $0.1637,0.0470,0.0141,0.0142,0.0039$ respectively, and for $(\vep,\delta) = (1/81,1/81)$ the relative error is $ 0.3608,0.0325,0.0176,0.0075,0.0125$ respectively.
		If the
		local map $\Scal_{m}^\text{s}$ admits a low rank structure, then the
		relative error is small even for a low-rank approximation.
	}
	\label{fig:global_err}
\end{figure}
\subsubsection*{Efficiency of approximating solution}

To demonstrate the efficiency of our method, we compare our reduced
Schwarz method and the ``vanilla'' Schwarz method (which does not use
low-rank approximations) in terms of accuracy and running time. In
particular, we run the reduced Schwarz method for $T = 50$ iterations,
and compare the number of iterations needed for vanilla Schwarz
method to achieve the same accuracy.
Figure~\ref{fig:time_array} plots relative error as a function of
iteration number $t$ of reduced Schwarz method and the vanilla
Schwarz, for parameter pairs $(\vep,\delta) = (1/81,1/81)$ and
$(\vep,\delta) = (1/81,1/9)$ and rank $r=6$. The convergence speed of
the vanilla and reduced versions of the Schwarz methods are quite
similar. However, the reduced Schwarz iteration is significantly
cheaper due to the use of the low rank structure. We document the run
time on a standard PC in Table~\ref{tbl:time}. We note that the
vanilla Schwarz iteration does not have the offline step, and the
online stage amounts to computing the equation for the given boundary
condition on each subdomain, and is very expensive. If we need to
solve the RTE for multiple boundary conditions, the computational
saving would be quite significant. We also report on results obtained
with the Schwarz procedure in which the full basis is prepared offline
(denoted in the table as ``Schwarz with full basis").

\begin{table}
	\centering
	\begin{tabular}{l | c | c |c | c}
		\hline \hline 
							& \multicolumn{4}{c}{Running Time (s)}\\
		\hline
							& \multicolumn{2}{c|}{$(\vep,\delta) = (1/81,1/81)$} & \multicolumn{2}{c}{$(\vep,\delta) = (1/81,1/9)$}\\
		\hline
		&offline& online & offline & online \\
		\hline
		Low-Rank Schwarz $r=2$	& 81.01		&	0.0022			& 107.22 &  0.0024 \\
		Low-Rank Schwarz $r=3$	& 111.43		&	0.0023			& 148.09 &   0.0022\\
		Low-Rank Schwarz $r=4$	& 139.97	&	0.0057		&  188.41&   0.0024\\
		Low-Rank Schwarz $r=5$	& 168.64	&	0.0032			& 227.99 &   0.0029\\
		Low-Rank Schwarz $r=6$	& 197.49	&	0.0065			& 268.46 &   0.0162\\
		\hline
		Schwarz with full basis&  535.26& 0.0061 &  706.99 &0.0148\\
		\hline
		Vanilla Schwarz & | &803.01&	|		&1027.40 \\
		\hline\hline 
	\end{tabular}
    \caption{Run time comparison between reduced Schwarz method with  $r=2,3,4,5,6$, an offline/online breakdown of Schwarz method and the vanilla Schwarz method.}
\label{tbl:time}
\end{table}

\begin{figure}
	\includegraphics[width=0.45\textwidth]{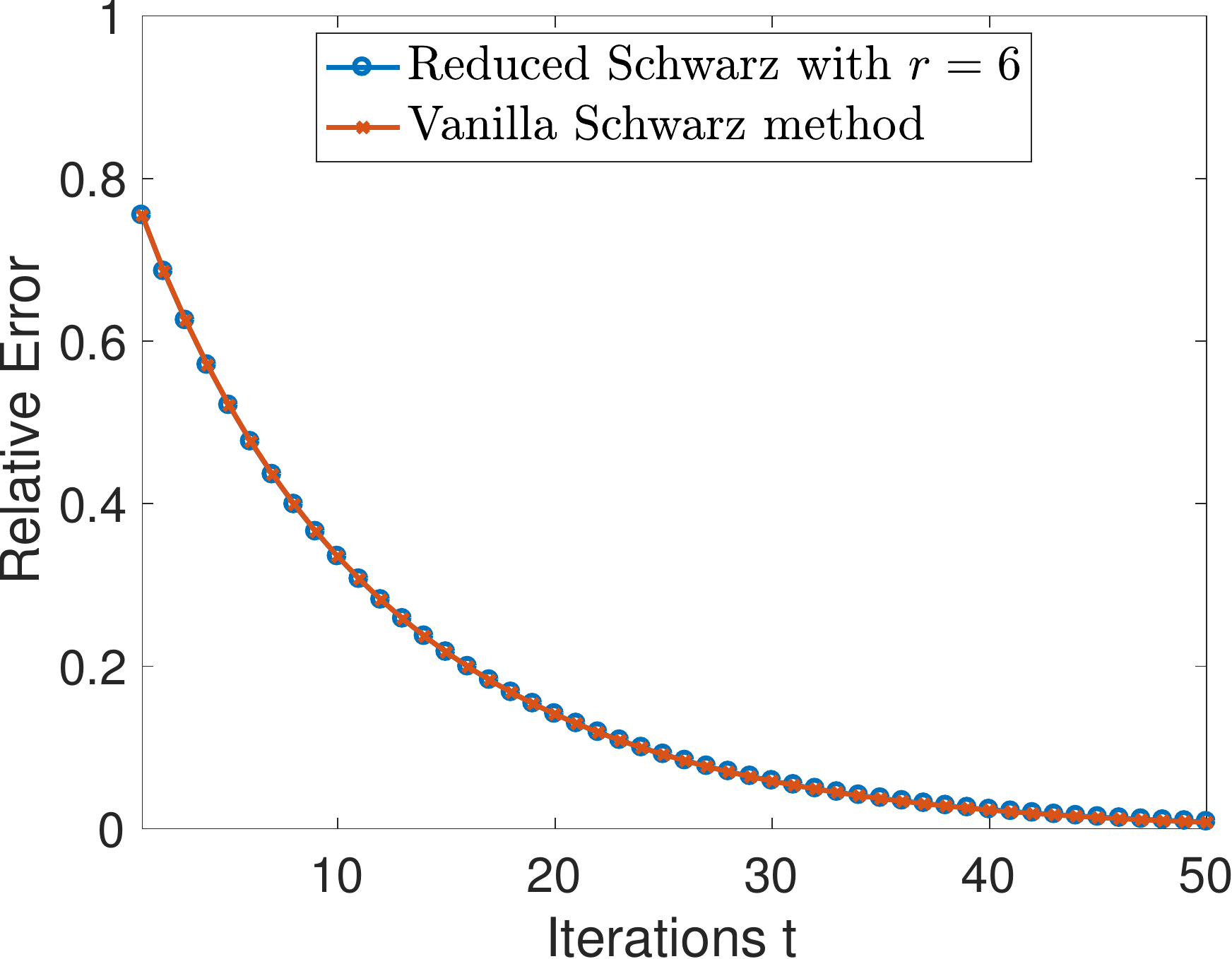}
	\includegraphics[width=0.45\textwidth]{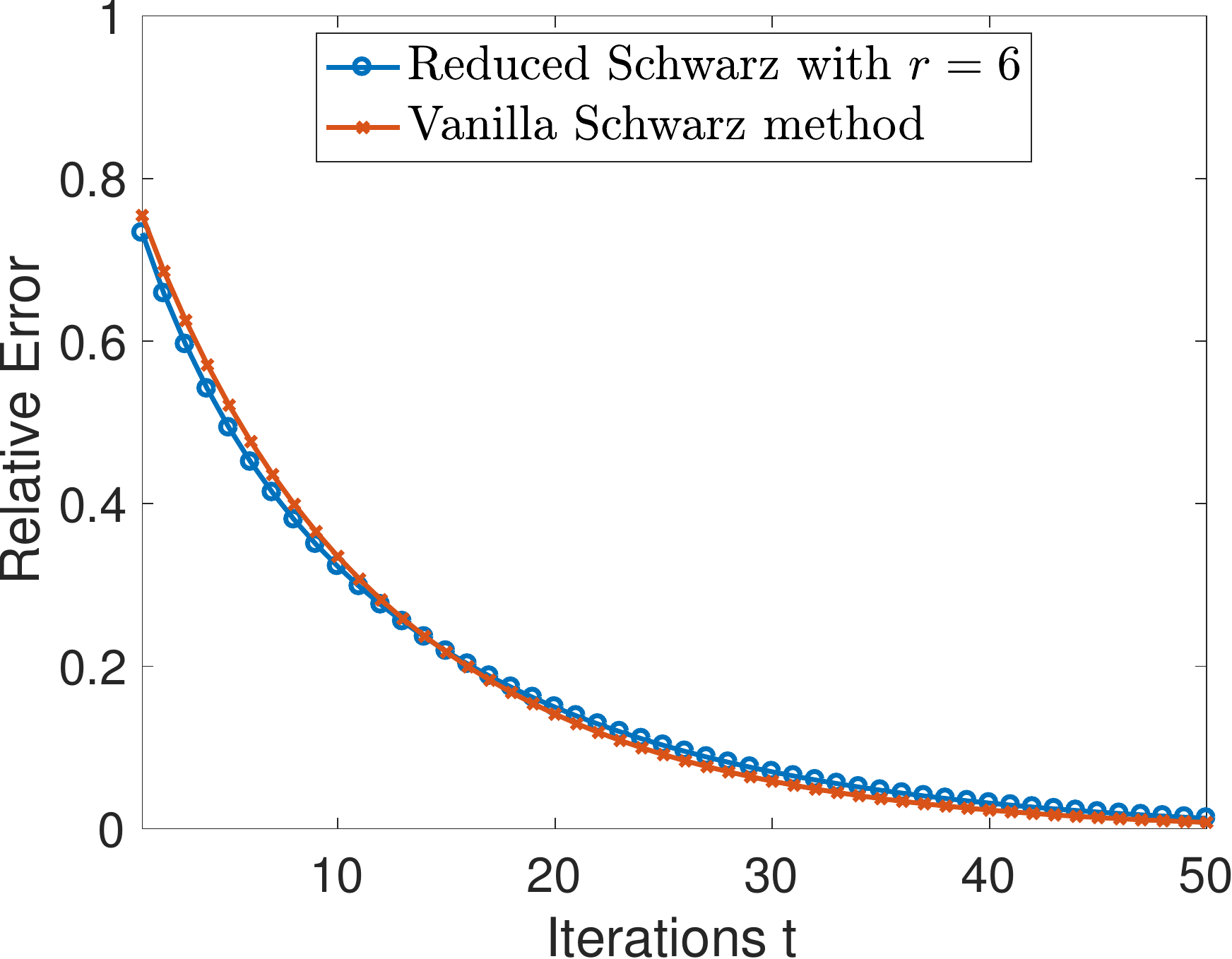}
	\caption{\textbf{Left:} time array of relative error for $(\vep,\delta)=(1/81,1/81)$; \textbf{Right:} time array of relative error for $(\vep,\delta)=(1/81,1/9)$
	}
	\label{fig:time_array}
\end{figure}

\section{Conclusion}

Random sampling is a popular technique in data
  science to find low-rank structure of a matrix. In particular, the
  randomized SVD method has proved to be an efficient and accurate
  technique for finding the dominant singular values and singular
  vectors of a matrix at reasonable cost.

Partial differential equations with multiscale
  structures can usuallly be described by an effective equation
  without fine scale details. In some sense, this phenomenon means the
  operator and the solution space are of low rank. We exploit this
  property to design efficient numerical schemes, using the radiative
  transfer equation as an example. Specifically, we utilize the
  Schwarz iteration in the domain-decomposition framework, with a
  boundary-to-boundary map that communicates between subdomains.  We
  use Randomized SVD to approximate this boundary-to-boundary map,
  exploiting the low-rank structure of the map. This computation is
  performed offline; the Schwarz iteration that makes use of these
  approximate maps is performed online.  Numerical examples confirm
  the effectiveness and computational efficiency of our approach.

Several aspects and extensions of our approach remain
  to be investigated. The biggest obstacle in making our approach
  fully rigorous is the lack of theoretical guarantees on the decay of
  singular values of the PDE operator, in either the $\epsilon\to0$ or
  the $\delta\to 0$ regimes. The lack of such guarantees makes
  numerical analysis hard to perform. Such results might point the way
  to better designs of the domain partition (especially the size of the
  buffer zones) and better choices of the accuracy threshold that
  defines the target rank $r$.

Extensions to spatial domains of dimension greater
  than $1$, especially in the choice of domain partitioning and
  computation of the boundary-to-boundary maps, remain a significant
  computational challenge that could be addressed in future work.

\section*{Acknowledgment}
The authors thank the two anonymous referees for
  insightful suggestions, which lead to great improvement of the
  paper.

\newpage 

\begin{appendix}

\section{Well-posedness Theory of RTE}

We show the well-posedness theory of RTE with fixed parameters
$\vep,\delta>0$ in this appendix. Most results are cited from
~\cite{agoshkov2012boundary}. For simplicity, we consider the
following RTE with inflow boundary condition:
\begin{equation}
	\begin{cases}
	v\cdot \nabla_x u = \Lcal u\,, \quad \text{in}\quad \Dcal\\
	u = \phi\,, \quad \text{on} \quad \Gamma_- \,.
	\end{cases}
\end{equation}
First, we introduce a functional space $H^1_2(\Dcal)$ with norm
defined as follows:
\begin{equation}\label{eqn:H^1_2}
	\| u \|_{H_2^1} = \left[\int_\Dcal |v\cdot \nabla_x u|^2 + |u|^2 \, \rd x \, \rd v \right]^{1/2}
\end{equation}
To find a suitable functional space for solutions of RTE, we modify
$H_2^1(\Dcal)$ and define a Hilbert space $H_A(\Dcal)$ with the
following scalar product and norm:
\begin{equation}\label{eqn:H_A}
	\langle u,w \rangle_A:= (u,w)_{H_2^1(\Dcal)} + \int_{\partial \Dcal} |n\cdot v| uw \, \rd x \, \rd v\,, \quad \|u \|_A = \langle u,u \rangle_A^{1/2}
\end{equation}
$H_A(\Dcal)$ is obviously a subspace of $H_2^1(\Dcal)$. The following
theorem shows well-posedness of RTE over $H_A(\Dcal)$:
\begin{theorem}[Theorem 3.7 in \cite{agoshkov2012boundary}]
	Given an inflow boundary condition $\phi \in L^2(\Gamma_-; |n\cdot v| )$, then there exists a unique solution $u \in H_A$ to RTE such that 
	\begin{equation}\label{thm:wellposed}
		c \| \phi \|_{L^2(\Gamma_-; |n\cdot v| )} \leq \|u\|_{H_A(\Dcal)} \leq \wt{c} \| \phi \|_{L^2(\Gamma_-; |n\cdot v| )}
	\end{equation}
\end{theorem}
Further, the trace operator $\gamma_{\pm}: u \in H_2^1(\Dcal) \mapsto u|_{\Gamma_{\pm}} \in L^2(\Gamma_\pm; |n\cdot v|)$ is also well-defined by the following theorem:
\begin{theorem}[Theorem 2.8 in \cite{agoshkov2012boundary}]
	If $u\in H_2^1(\Dcal)$, then $u$ has a trace over $\Gamma_\pm$ belonging to $L^2(\Gamma_\pm; |n\cdot v|)$. In addition, we have 
	\begin{equation}\label{eqn:trace}
		\| u|_{\Gamma_{\pm}} \|_{L^2(\Gamma_\pm; |n\cdot v|)} \leq c \| u \|_{H_2^1(\Dcal)}
	\end{equation}
\end{theorem}
Moreover, for RTE with a source term:
\begin{equation}
	\begin{cases}
	v\cdot \nabla_x u = \Lcal u + f \quad & \text{in}\quad \Dcal\\
	u = \phi \quad & \text{on} \quad \Gamma_- 
	\end{cases}
\end{equation}
the following theorem holds.
\begin{theorem}[Theorem 3.13 in \cite{agoshkov2012boundary}]
  If $f$ belongs to $ H_A^{-1}(\Dcal)$, the dual space $H_A$
  under $L^2$ pairing over $\Dcal$, $\phi \in L^2(\Gamma_-;
  |n\cdot v| )$ then the above equation admits an unique
  solution $u\in L^2(\Dcal)$ such that
  \[
  \| u \|_{L^2(\Dcal)} \leq c\left[\|f \|_{H_A^{-1}(\Dcal)} + \| \phi\|_{L^2(\Gamma_-; |n\cdot v| )} \right].
  \]
  Further, if $f\in L^2(\Dcal)$, then the solution $u\in H_A(\Dcal)$.
\end{theorem}
\begin{remark}
	The theorems above also hold for adjoint RTE with or without a source term. Details can be found in \cite{agoshkov2012boundary}.
\end{remark}
\subsection*{Proof of Theorem~\ref{thm:adjoint}}
\begin{proof}
Considering any $\phi\in L^2(\Gamma_{m,-}; |n\cdot v| ) $ and $\psi
\in L^2(\Gamma^s_{m,+}; |n\cdot v|)$, we have
\[
	\langle \Pcal_m\phi,\psi \rangle_{\Gamma_{m,+}^s} = \int_{\Gamma_{m,+}^s} f \psi (n\cdot v) = \int_{\Gamma_{m,+}^s} fg(n\cdot v) - \int_{\Gamma^s_{m,+}} fh (n\cdot v),
\]
where the first equality comes from definition of $\Pcal_m$ and the
second from the definition of $g$. By multiplying
\eqref{eqn:adjoint_RTE1} by $f$ and integrating over $\Dcal_m^s$, we
obtain
\[
0 = \int_{\Dcal_m^s} f(-v\cdot \nabla_x -\frac{\sigma}{\vep}\Lcal ) g
= \int_{\Gamma^s_{m,+}} fg(n\cdot v) + \int_{\Gamma^s_{m,-}}fg(-n\cdot
v),
\]
where the first equality comes from \eqref{eqn:adjoint_RTE1} and
second from integration by parts. Comparing the equations above, we
have
\[
\langle \Pcal_m\phi,\psi \rangle_{\Gamma_{m,+}^s} =
\int_{\Gamma^s_{m,-}}fg(n\cdot v)- \int_{\Gamma^s_{m,+}} fh (n\cdot
v).
\]
It is easy to see that from the definition of $h$ we also have
\[
	\langle \Pcal_m\phi,\psi \rangle_{\Gamma_{m,+}^s} =
        \int_{\Gamma^s_{m,-}}fh(n\cdot v)- \int_{\Gamma^s_{m,+}} fh
        (n\cdot v).
\]
By multiplying \eqref{eqn:adjoint_RTE2} by $f$ and integrating over
$\Dcal_m\backslash \Dcal_m^s$, we obtain
\begin{align*}
    0 &= \int_{\Dcal/\Dcal_m^s} f(-v\cdot \nabla_x -\frac{\sigma}{\vep}\Lcal ) h \\
    & = \int_{\Gamma^s_{m,+}} fh(\tilde{n}\cdot v) + \int_{\Gamma^s_{m,-}}fh(-\tilde{n}\cdot v) + \int_{\Gamma_{m,-}} fh(-\tilde{n}\cdot v) + \int_{\Gamma_{m,+}} fh(\tilde{n}\cdot v),
\end{align*}
where the first equality comes from \eqref{eqn:adjoint_RTE2} and
second from integration by parts. We use notation $\tilde{n}$ as the
outer normal direction over $\Gamma^s_{m,+}$ and $\Gamma^s_{m,-}$ with
respect to the domain $\Dcal_m\backslash \Dcal_m^s$, to distinguish it
from the outer normal with respect to the domain $\Dcal_m^s$. In fact,
the two instances of ``outer normal'' have opposite directions when we
interpret $\Gamma^s_{m,\pm}$ as the boundary of $\Dcal_m^s$ and the
boundary of $\Dcal_m\backslash \Dcal_m^s$. By comparing the equations
above, we have
\[
\langle \Pcal_m\phi,\psi \rangle_{\Gamma_{m,+}^s} =\int_{\Gamma_{m,-}}
fh(-\tilde{n}\cdot v) + \int_{\Gamma_{m,+}} fh(\tilde{n}\cdot v).
\]
Noticing that $h = 0 $ over $\Gamma_{m,+}$ and $\tilde{n}$ is also the
outer normal direction over $\Gamma_{m,-}$ when interpreted as the
boundary of $\Dcal_m$, we have
\[
	\langle \Pcal_m\phi,\psi \rangle_{\Gamma_{m,+}^s}
        =\int_{\Gamma_{m,-}} fh(-\tilde{n}\cdot v) =
        \int_{\Gamma_{m,-}} fh(-n\cdot v) = \langle \phi, \Ycal_m \psi
        \rangle_{\Gamma^s_{m,+}},
\]
where the last equality comes from the definition of $f$ and $\psi$.
\end{proof}

\end{appendix}

\newpage
\bibliographystyle{siam}
\bibliography{ref}

\end{document}